\documentclass[a4paper]{amsart}
\usepackage[english]{babel}
\usepackage[utf8]{inputenc}
\usepackage[T1]{fontenc}
\usepackage{latexsym}
\usepackage{amsmath}
\usepackage{amscd}
\usepackage{amsfonts}
\usepackage{amsthm}
\usepackage{amssymb}
\usepackage[dvips]{pstricks}
\usepackage{amsaddr}
\usepackage{mathrsfs}
\usepackage{textcomp}
\usepackage{hyperref}
 
\numberwithin{equation}{section}
\newtheorem{thm}{Theorem}
\newtheorem{thmbis}{Theorem}
\newtheorem{proposition}{Proposition}[section]
\newtheorem{corollary}[proposition]{Corollary}
\newtheorem{lemma}[proposition]{Lemma}
\newtheorem{property}[proposition]{Property}

\newtheorem*{strongMarkov}{Strong Markov property}

\begin{document}

\title[Loop clusters, random interlacements and GFF]{From loop clusters and random interlacements to the free field}
\author{Titus Lupu}
\address{Laboratoire de Mathématiques, Université Paris-Sud, Orsay}
\email{titus.lupu@math.u-psud.fr}
\thanks{This research was funded by the PhD contract \textnumero
2013-125 with Université Paris-Sud, Orsay.}
\subjclass[2010]{60K35, 60G15, 60G60, 60J25}
\keywords{Gaussian free field, loop soup, Poisson ensemble of Markov loops, percolation by loops,  random interlacements}

\begin{abstract}
It was shown by Le Jan that the occupation field of a Poisson ensemble of Markov loops ("loop soup") of parameter $\frac{1}{2}$ associated to a transient symmetric Markov jump process on a network is half the square of the Gaussian free field on this network. We construct a coupling between these loops and the free field such that an additional constraint holds: the sign of the free field is constant on each cluster of loops. As a consequence of our coupling we deduce that the loop clusters of parameter $\frac{1}{2}$ do not percolate on periodic lattices. We also construct a coupling between the random interlacement on $\mathbb{Z}^{d}$, $d\geq 3$, introduced by Sznitman, and the Gaussian free field on the lattice such that the set of vertices visited by the interlacement is contained in a one-sided level set of the free field. We deduce an inequality between the critical level for the percolation by level sets of the free field and the critical parameter for the percolation of the vacant set of the random interlacement. 
\end{abstract}

\maketitle

\section{Introduction}
\label{SecInro}

Here we introduce our framework, some notations, state our main results and outline the layout of the paper.

We consider a connected undirected graph $\mathcal{G}=(V,E)$ where the set of vertices $V$ is at most countable and every vertex has finite degree. We do not allow multiple edges nor loops from a vertex to itself. The edges are endowed with positive conductances $(C(e))_{e\in E}$ and vertices endowed with a non-negative killing measure 
$(\kappa(x))_{x\in V}$. $\kappa$ may be uniformly zero. 
$(X_{t})_{0\leq t<\zeta}$ is a continuous-time sub-Markovian jump process on $V$. Given two neighbouring vertices $x$ and $y$, the transition rate from $x$ to $y$ equals the conductance $C(x,y)$. Moreover there is a transition rate $\kappa(x)$ from $x\in V$ to a cemetery point outside $V$. Once such a transition occurs, the process $X$ is considered to be killed. Moreover we allow $X$ to blow up in finite time, i.e. leave all finite sets. $\zeta$ is either $+\infty$ or the first time $X$ gets killed or blows up. We assume that $X$ is transient, which is a condition on $C$ and $\kappa$. 
In particular if $\kappa$ is not uniformly zero, $X$ is transient. 
$(G(x,y))_{x,y\in V}$ denotes the Green's function of $X$,
\begin{displaymath}
G(x,y)=\mathbb{E}_{x}\bigg[\int_{0}^{\zeta}1_{X_{t}=y} dt\bigg],
\end{displaymath}
and $G$ is symmetric.

Let $\left(\mathbb{P}^{t}_{x,y}(\cdot)\right)_{x,y\in V, t>0}$
be the bridge probability measures of $X$, conditioned on $\zeta >t$ and let $(p_{t}(x,y))_{x,y\in V, t\geq 0}$ be the transition probabilities of $X$. The measure $\mu$ on time-parametrized loops associated to $X$ is, as defined in \cite{LeJan2011Loops},
\begin{equation}
\label{IntroDefmu}
\mu(\cdot)=\sum_{x\in V}\int_{0}^{+\infty}
\mathbb{P}^{t}_{x,x}(\cdot)p_{t}(x,x)\dfrac{dt}{t} .
\end{equation}
Let $\alpha >0$. $\mathcal{L}_{\alpha}$ is defined to be the Poisson point process in the space of loops on $\mathcal{G}$ with intensity $\alpha\mu$. It is sometimes called loop-soup of parameter $\alpha$. The occupation field $(\widehat{\mathcal{L}}^{x}_{\alpha})_{x\in V}$ of $\mathcal{\mathcal{L}}_{\alpha}$ is
\begin{displaymath}
\widehat{\mathcal{L}}^{x}_{\alpha}
=\sum_{\gamma\in\mathcal{L}_{\alpha}}
\int_{0}^{T(\gamma)}1_{\gamma(t)=x}dt,
\end{displaymath}
where $T(\gamma)$ is the duration of the loop $\gamma$.
The loops of $\mathcal{L}_{\alpha}$ may be partitioned into clusters: if $\gamma,\gamma'\in\mathcal{L}_{\alpha}$ belong to the same cluster if there is a chain $\gamma_{0},\dots,\gamma_{n}$ of loops in $\mathcal{L}_{\alpha}$ such that $\gamma_{0}=\gamma$, $\gamma_{n}=\gamma'$ and for all $i\in \lbrace 1,\dots,n\rbrace$
$\gamma_{i-1}$ and $\gamma_{i}$ visit a common vertex
(\cite{LeJanLemaire2012LoopClusters}). A cluster $\mathcal{C}$ is a set of loops, but it also induces a sub-graph of 
$\mathcal{G}$. Its vertices are the vertices of $\mathcal{G}$ visited by at least one loop in $\mathcal{C}$ and its edges are those that join two consecutive points of a loop in $\mathcal{C}$. Therefore we will also consider $\mathcal{C}$ as a subset of vertices and a subset of edges and use the notations $\gamma\in\mathcal{C}$, $x\in\mathcal{C}$ and
$e\in\mathcal{C}$ where $\gamma$ is a loop, $x$ is a vertex and $e$ is an edge. $\mathfrak{C}_{\alpha}$ will be the random set of all clusters of $\mathcal{L}_{\alpha}$. It induces a partition of $V$.

Let $(\phi_{x})_{x\in V}$ be the Gaussian free field on $\mathcal{G}$, i.e. the mean-zero Gaussian field with 
$\mathbb{E}[\phi_{x}\phi_{y}]=G(x,y)$. In \cite{LeJan2011Loops}, Section $5$, Le Jan showed that at intensity parameter
$\alpha=\frac{1}{2}$ the occupation field 
$(\widehat{\mathcal{L}}^{x}_{1/2})_{x\in V}$ has the same law as
$(\frac{1}{2}\phi^{2}_{x})_{x\in V}$. This equality in law may be seen as an extension of Dynkin's isomorphism (\cite{Dynkin1984Isomorphism},
\cite{Dynkin1984IsomorphismPresentation}) and in turn enables an alternative derivation of some version of Dynkin's isomorphism through the use of Palm's identity for Poisson point processes (\cite{LeJanMarcusRosen2012Loops},\cite{FitzsimmonsRosen2012LoopsIsomorphism} and \cite{Lupu20131dimLoops}, Section $4.3$). However the question of relating the sign of $\phi$ to
$\mathcal{L}_{1/2}$ remained open. In this paper we show the following:

\begin{thm}
\label{ThmCoupling}
There is a coupling between the Poisson ensemble of loops 
$\mathcal{L}_{1/2}$ and the Gaussian free field $\phi$ such that the 
following two constraints hold:
\begin{itemize}
\item for all $x\in V$, 
$\widehat{\mathcal{L}}^{x}_{1/2}=\frac{1}{2}\phi_{x}^{2}$;
\item for all $\mathcal{C}\in\mathfrak{C}_{1/2}$ the sign of $\phi$ is constant on the vertices of $\mathcal{C}$.
\end{itemize}
\end{thm}

In Section \ref{SecCouplComplex} we will construct the coupling that satisfies the constraints of Theorem \ref{ThmCoupling}. To this end we will introduce the metric graph $\widetilde{\mathcal{G}}$ associated to the graph $\mathcal{G}$ and interpolate the loops in
$\mathcal{L}_{1/2}$ by continuous loops on $\widetilde{\mathcal{G}}$. In Section \ref{SecAltDesc} we will show that the same coupling can be described without using the metric graph $\widetilde{\mathcal{G}}$ and the interpolation of loops:

\begin{thmbis}
\label{ThmAltDescCoupl}
\textbf{\emph{bis}}.
Consider the following construction:
\begin{itemize}
\item First sample the Poisson ensemble of loops 
$\mathcal{L}_{1/2}$ with $(\widehat{\mathcal{L}}^{x}_{1/2})_{x\in V}$ being its occupation field and 
$\mathfrak{C}_{1/2}$ the set of its clusters.
\item For any edge $\lbrace x,y\rbrace$ not visited by any loop in 
$\mathcal{L}_{1/2}$, choose to open it with probability 
$1-\exp\Big(-2C(x,y)\sqrt{\widehat{\mathcal{L}}^{x}_{1/2}
\widehat{\mathcal{L}}^{y}_{1/2}}\Big)$. By doing so some cluster of $\mathfrak{C}_{1/2}$ may merge and this induces a partition 
$\mathfrak{C}'$ of $V$ in larger clusters.
\item For all clusters $\mathcal{C}'\in\mathfrak{C}'$ sample independent uniformly distributed in $\lbrace -1,+1\rbrace$ r.v.'s 
$\sigma(\mathcal{C}')$.
\item Set $\phi_{x}:=\sigma(\mathcal{C}'(x))
\sqrt{2\widehat{\mathcal{L}}^{x}_{1/2}}$ where $\mathcal{C}'(x)$ is the cluster in $\mathfrak{C}'$ containing the vertex $x$.
\end{itemize}
$(\phi_{x})_{x\in V}$ is then a Gaussian free field on $\mathcal{G}$. Moreover the obtained coupling between $\mathcal{L}_{1/2}$ and 
$\phi$ is the same, in law, as the one constructed in Section \ref{SecCouplComplex}.
\end{thmbis}

In Section \ref{SecAltProof} we will give an alternative, direct, proof that the coupling holds using its description given in Section \ref{SecAltDesc}. 

In Section \ref{SecPerco} we will apply Theorem \ref{ThmCoupling} to the loop percolation problem. The loops of $\mathcal{L}_{1/2}$ are said to percolate if there is an unbounded cluster of loops. This question percolation was studied in \cite{LeJanLemaire2012LoopClusters} and \cite{ChangSapozhnikov2014PercLoops}. Obviously from Theorem \ref{ThmCoupling} follows that the loops do not percolate if the sign clusters of $\phi$ are all bounded. But we will show that even in some situations where $\phi$ is known to have some (two) infinite sign clusters, the loops of $\mathcal{L}_{1/2}$ still do not percolate:

\begin{thm}
\label{ThmPerco}
Consider the following networks:
\begin{itemize}
\item $\mathbb{Z}^{2}$ with uniform conductances and a nonzero uniform killing measure;
\item the discrete half-plane $\mathbb{Z}\times\mathbb{N}$ with instantaneous killing on the boundary $\mathbb{Z}\times\lbrace 0\rbrace$
and no killing elsewhere;
\item $\mathbb{Z}^{d}$, $d\geq 3$, with uniform conductances and no killing measure.
\end{itemize}
On all above networks $\mathcal{L}_{1/2}$ does not percolate.
\end{thm}

We will also give a bound for the probability that two vertices belong to the same cluster of loops.

In Section \ref{SecInterlacement} we consider random interlacements on $\mathbb{Z}^{d}$ introduced by Sznitman 
(\cite{Sznitman2010RandomInterlacement}). We consider that the edges of $\mathbb{Z}^{d}$ have conductances equal to $1$ and that 
$(G(x,y))_{x,y\in\mathbb{Z}^{d}}$ and $(\phi_{x})_{x\in\mathbb{Z}^{d}}$ are the corresponding Green's function and Gaussian free field. Given $K$ a finite subset of $\mathbb{Z}^{d}$, let $e_{K}$ be the equilibrium measure of $K$ (supported on $K$):
\begin{displaymath}
\forall x\in K, e_{K}(\lbrace x\rbrace)=\mathbb{P}_{x}(\forall j\geq 1, Y_{j}\not\in K),
\end{displaymath}
where $(Y_{j})_{j\geq 0}$ is the simple random walk on
$\mathbb{Z}^{d}$.
The capacity of $K$ is
\begin{displaymath}
\operatorname{cap}(K)=e_{K}(K).
\end{displaymath}
Let $Q_{K}$ be the measure on doubly infinite trajectories on $\mathbb{Z}^{d}$, $(x_{j})_{j\in\mathbb{Z}}$ parametrized by discrete time $j\in\mathbb{Z}$, of total mass $\operatorname{cap}(K)$, such that
\begin{itemize}
\item the measure on $x_{0}$ induced by $Q_{K}$ is $e_{K}$;
\item conditional on $x_{0}$, $(x_{j})_{j\geq 0}$ and 
$(x_{j})_{j\leq 0}$ are independent;
\item conditional on $x_{0}$, $(x_{j})_{j\geq 0}$ is a nearest neighbour random walk on $\mathbb{Z}^{d}$ starting from $x_{0}$;
\item conditional on $x_{0}$, $(x_{-j})_{j\geq 0}$ is a nearest neighbour random walk on $\mathbb{Z}^{d}$ starting from $x_{0}$ conditioned not to return in $K$ for $j\geq 1$.
\end{itemize}
There is an (infinite) measure $\mu_{il}$ on right continuous doubly infinite trajectories $(w(t))_{t\in\mathbb{R}}$ on $\mathbb{Z}^{d}$, parametrized by continuous time, considered up to a translation of parametrization 
($(w(t))_{t\in\mathbb{R}}$ same as $(w(t))_{t+t_{0}\in\mathbb{R}}$) such that
\begin{itemize}
\item $\lim_{t\rightarrow+\infty}\vert w(t)\vert = \lim_{t\rightarrow-\infty}\vert w(t)\vert=+\infty$ $\mu_{il}$-almost everywhere;
\item for any finite subset $K$ of $\mathbb{Z}^{d}$, 
by restricting $\mu_{il}$ to trajectories visiting $K$, choosing the initial time $t=0$ to be the first entrance time in $K$ and taking the skeleton (the doubly infinite sequence of successively visited vertices) we get the measure $Q_{K}$;
\item under $\mu_{il}$, conditional on the skeleton, the doubly infinite sequence of holding times of the trajectory (times spent at vertices before jumping to neighbours) is i.i.d with exponential distribution of mean $(2d)^{-1}$. 
\end{itemize}
See \cite{Sznitman2010RandomInterlacement} and \cite{Sznitman2012Isomorphism}.

The random interlacement $\mathcal{I}^{u}$ of level $u>0$ is the Poisson point process of intensity $u\mu_{il}$.
The vacant set $\mathcal{V}^{u}$ of $\mathcal{I}^{u}$ is the set of vertices not visited by any of trajectories in $\mathcal{I}^{u}$. There is $u_{\ast}\in (0,+\infty)$ such that for $u<u_{\ast}$, 
$\mathcal{V}^{u}$ has a.s. infinite connected components and for $u>u_{\ast}$ $\mathcal{V}_{u}$ has a.s. only finite connected components (\cite{Sznitman2010RandomInterlacement}, \cite{SidoraviciusSznitman2009VacantSet}).

The occupation field $(L^{x}(\mathcal{I}^{u}))_{x\in\mathbb{Z}^{d}}$ of the interlacement $\mathcal{I}^{u}$ is defined as
\begin{displaymath}
L^{x}(\mathcal{I}^{u}):=\sum_{w\in \mathcal{I}_{u}}\int_{-\infty}^{+\infty}1_{w(t)=x} dt.
\end{displaymath}
In \cite{Sznitman2012Isomorphism} Sznitman showed the following isomorphism between 
$(L^{x}(\mathcal{I}^{u}))_{x\in\mathbb{Z}^{d}}$ and the Gaussian free field: Let 
$(\phi'_{x})_{x\in\mathbb{Z}^{d}}$
be a copy of the free field independent of $(L^{x}(\mathcal{I}^{u}))_{x\in\mathbb{Z}^{d}}$. Then
\begin{equation}
\label{SecIntro: EqIsoSz}
\Big(L^{x}(\mathcal{I}^{u})+\dfrac{1}{2}\phi'^{2}_{x}\Big)_{x\in\mathbb{Z}^{d}}\stackrel{(d)}{=}
\Big(\dfrac{1}{2}(\phi_{x}-\sqrt{2u})^{2}\Big)_{x\in\mathbb{Z}^{d}}.
\end{equation}
This isomorphism can be used to relate the random interlacement to the level sets of the Gaussian free field. There is $h_{\ast}\in[0,+\infty)$ such that for $h<h_{\ast}$, the set $\lbrace x\in \mathbb{Z}^{d}\vert \phi_{x}>h\rbrace$ has an infinite connected components and for $h>h_{\ast}$ only finite connected components (\cite{RodriguezSznitman2013PercGFF},
\cite{BricmontLebowitzMaes1987PercGFF}). $h_{\ast}$ is positive if the dimension $d$ high enough (\cite{RodriguezSznitman2013PercGFF}). In Section \ref{SecInterlacement} we will prove:

\begin{thm}
\label{ThmLevelSets}
For all $u>0$, there is a coupling between $\mathcal{I}^{u}$ and $\phi$ such that a.s.
\begin{displaymath}
\lbrace x\in \mathbb{Z}^{d}\vert \phi_{x}>\sqrt{2u}\rbrace\subseteq\mathcal{V}^{u}.
\end{displaymath}
In particular
\begin{displaymath}
h_{\ast}\leq \sqrt{2u_{\ast}}.
\end{displaymath}
\end{thm}

This theorem is again obtained by replacing the discrete graph $\mathbb{Z}^{d}$ by a metric graph.

\section{Coupling through interpolation by a metric graph}
\label{SecCouplComplex}

One can associate a measure on loops following the formal pattern of 
\eqref{IntroDefmu} to a wide range of Markovian or sub-Markovian processes. In the articles \cite{LeJanMarcusRosen2012Loops} and \cite{FitzsimmonsRosen2012LoopsIsomorphism} the authors give quite general definitions for a wide range of cases. The setting of \cite{FitzsimmonsRosen2012LoopsIsomorphism}
will cover our needs. In that article the measure on loops is defines for transient Borel right processes on a locally compact state space with with a countable base, that have $0$-potential densities with respect some sigma-finite measure, the $0$-potential densities being assumed to be finite everywhere (in particular on the diagonal) and continuous.
In \cite{Lupu20131dimLoops} were specifically studied the measures on loops associated to one-dimensional diffusions and the corresponding loop ensembles. This case is of particular interest for the proof of Theorem
\ref{ThmCoupling}. Indeed in the setting of one-dimensional diffusions the occupation fields are continuous space-parametrized processes with non-negative values and the clusters of loops correspond exactly to the excursions of the occupation field above zero (Proposition $4.8$ in \cite{Lupu20131dimLoops}). In particular for the loop ensemble of parameter $\frac{1}{2}$, the clusters of loops are exactly the sign clusters of the one-dimensional Gaussian free field.

The nice identity between the clusters of loops and the sign clusters of 
GFF in case of one-dimensional diffusions leads us to consider the metric graph or cable system $\widetilde{\mathcal{G}}$ associated to the graph $\mathcal{G}$ (\cite{BaxterChacon1984DiffusionsNetworks}, \cite{EnriquezKifer2001BMGraphs},
\cite{Folz2014VolGrowthStochCompl}). Topologically $\widetilde{\mathcal{G}}$ is constructed as follows: to each edge $e$ of $\mathcal{G}$ corresponds a different compact interval, each endpoint of this interval being identified to one of the two vertices adjacent to $e$ in $\mathcal{G}$; for every vertex $x\in V$ the intervals corresponding to the edges adjacent to $x$ are glued together at the endpoints identified to the vertex $x$. We will consider $V$ to be a subset of $\widetilde{\mathcal{G}}$. Given any $e\in E$, $I_{e}$ will denote the subset of $\widetilde{\mathcal{G}}$ made of the interval corresponding to $e$ minus its two endpoints. Topologically $I_{e}$ is an open interval. $\widetilde{\mathcal{G}}$ is a disjoint union
\begin{displaymath}
\widetilde{\mathcal{G}}=V\cup\bigcup_{e\in E} I_{e}.
\end{displaymath}
We further endow $\widetilde{\mathcal{G}}$ with a metric structure by assigning a finite length to each of the $(I_{e})_{e\in I}$. The length of $I_{e}$ is set to be
\begin{displaymath}
\rho(e):=\dfrac{1}{2C(e)},
\end{displaymath}
which makes $I_{e}$ isometric to $(0,\rho(e))$.  This particular choice of the lengths will be explained farther. Let $m$ be the Borel measure on $\widetilde{\mathcal{G}}$ assigning a zero mass to $V$, a mass $\rho(e)$ to each of the $I_{e}$ and to a subinterval of $I_{e}$ a mass equal to its length. $m$ is $\sigma$-finite.

On $\widetilde{\mathcal{G}}$ one can define a standard Brownian motion
$B^{\widetilde{\mathcal{G}}}$. Here we give a description through chaining stopped Markovian paths on $\widetilde{\mathcal{G}}$ (see 
\cite{BaxterChacon1984DiffusionsNetworks}, \cite{EnriquezKifer2001BMGraphs} and \cite{Folz2014VolGrowthStochCompl}). If $B^{\widetilde{\mathcal{G}}}$ starts in the interior $I_{e}$ of an edge, it behaves as the standard Brownian motion on $I_{e}$ until it reaches a vertex. To describe the behaviour of $B^{\widetilde{\mathcal{G}}}$ starting from a vertex we use the excursions. Let $x_{0}\in V$, 
$\lbrace x_{1},\dots,x_{\deg(x_{0})}\rbrace$ the vertices adjacent to $x_{0}$ and $\lbrace\lbrace x_{0},x_{1}\rbrace,\dots,
\lbrace x_{0},x_{\deg(x_{0})}\rbrace\rbrace$ the edges joining $x_{0}$ to one of its neighbours. Let $(B_{t})_{t\geq 0}$ be a standard Brownian motion on $\mathbb{R}$ starting from $0$. To each excursion $\mathtt{e}$ of $(B_{t})_{t\geq 0}$ away from $0$ we associate a random variable $x(\mathtt{e})$ uniformly distributed in $\lbrace x_{1},\dots,x_{\deg(x_{0})}\rbrace$. We chose the different r.v.'s $x(\mathtt{e})$ to be independent conditional on the family of excursions of $(B_{t})_{t\geq 0}$. $\mathtt{e}_{t}$ the excursion straddling the time $t$. Let be
\begin{displaymath}
T_{\lbrace x_{1},\dots,x_{\deg(x_{0})}\rbrace}:=
\inf\lbrace t\geq 0\vert \vert B_{t}\vert\geq 
\rho(\lbrace x_{0},x(\mathtt{e}_{t})\rbrace)\rbrace.
\end{displaymath}
To the path $(B_{t})_{0\leq t\leq T_{\lbrace x_{1},\dots,x_{\deg(x_{0})}\rbrace}}$ we associate a path in $\widetilde{\mathcal{G}}$: it starts at $x_{0}$ and each excursion 
$\mathtt{e}$ of $(B_{t})_{0\leq t\leq T_{\lbrace x_{1},\dots,x_{\deg(x_{0})}\rbrace}}$ is performed in 
$I_{\lbrace x_{0},x(\mathtt{e})\rbrace}$ instead of $\mathbb{R}$. The obtained path has the law of $B^{\widetilde{\mathcal{G}}}$ starting at $x_{0}$ and stopped at reaching 
$\lbrace x_{1},\dots,x_{\deg(x_{0})}\rbrace$. Let $(L^{y}_{t}(B))_{t\geq 0, y\in\mathbb{R}}$ be the continuous family of local times of $B$ and 
$(L^{y}_{t}(B^{\widetilde{\mathcal{G}}}))_{t\geq 0, y\in\widetilde{\mathcal{G}}}$ the family of local times of 
 $B^{\widetilde{\mathcal{G}}}$ started at $x_{0}$, relative to the measure $m$. Let $y\in I_{\lbrace x_{0},x_{i}\rbrace}$ and $\delta$ be the length of the subinterval $(x_{0},y)$ of 
$I_{\lbrace x_{0},x_{i}\rbrace}$. Then
\begin{multline*}
(L^{y}_{t}(B^{\widetilde{\mathcal{G}}}))
_{0\leq t\leq T_{\lbrace x_{1},\dots,x_{\deg(x_{0})}\rbrace}}\\=
\left(\int_{0}^{t}1_{x(\mathtt{e}_{s})=x_{i}}
(dL^{\delta}_{s}(B)+dL^{-\delta}_{s}(B))\right)
_{0\leq t\leq T_{\lbrace x_{1},\dots,x_{\deg(x_{0})}\rbrace}},
\end{multline*}
and the limit, uniform in time, of the above process as $y$ converges to $x_{0}$ is
\begin{displaymath}
\left(\dfrac{2}{\deg(x_{0})}L^{0}_{t}(B)\right)
_{0\leq t\leq T_{\lbrace x_{1},\dots,x_{\deg(x_{0})}\rbrace}},
\end{displaymath}
whatever the value of $i$. Let $\mathcal{B}(x_{0},\delta)$ be the ball in $\widetilde{\mathcal{G}}$ around $x_{0}$ of radius $\delta$. If $\delta\leq\min_{1\leq i\leq\deg(x_{0})}
\rho(\lbrace x_{0},x_{i}\rbrace)$ then 
$m(\mathcal{B}(x_{0},\delta))=\deg(x_{0})\delta$. It follows that for 
$t\in [0,T_{\lbrace x_{1},\dots,x_{\deg(x_{0})}\rbrace}]$,
\begin{equation*}
\begin{split}
\lim_{\delta \rightarrow 0}\dfrac{1}{m(\mathcal{B}(x_{0},\delta))}
\int_{0}^{t}1_{B^{\widetilde{\mathcal{G}}}_{s}\in 
\mathcal{B}(x_{0},\delta)} ds =&\lim_{\delta \rightarrow 0}
\dfrac{1}{\deg(x_{0})\delta}\int_{0}^{t}1_{\vert B_{s}\vert<\delta} ds
\\=&\dfrac{2}{\deg(x_{0})}L^{0}_{t}(B).
\end{split}
\end{equation*}
It follows that the process $(B^{\widetilde{\mathcal{G}}}_{t})_{0\leq t\leq T_{\lbrace x_{1},\dots,x_{\deg(x_{0})}\rbrace}}$ has a space-time continuous family of local times. By concatenating different stopped paths  we get that the whole process $B^{\widetilde{\mathcal{G}}}$ has space-time continuous local times relative to the measure $m$. 
The measure on the height of excursions (in absolute value) induced by the measure on Brownian excursions is (see \cite{RevuzYor1999BMGrundlehren}, Chapter XII, Section $4$)
\begin{displaymath}
1_{a>0}\dfrac{da}{a^{2}}.
\end{displaymath}
It follows that 
$L^{0}_{T_{\lbrace x_{1},\dots,x_{\deg(x_{0})}\rbrace}}(B)$ is an exponential random variable with mean
\begin{displaymath}
\dfrac{\deg(x_{0})}{\sum_{i=1}^{\deg x_{0}}
\rho(\lbrace x_{0},x_{i}\rbrace)^{-1}}=
\dfrac{\deg(x_{0})}{2\sum_{i=1}^{\deg x_{0}}C(x_{0},x_{i})}.
\end{displaymath}
$L^{x_{0}}_{T_{\lbrace x_{1},\dots,x_{\deg(x_{0})}\rbrace}}(B^{\widetilde{\mathcal{G}}})$ is an exponential random variable with mean
\begin{displaymath}
\dfrac{1}{\sum_{i=1}^{\deg x_{0}}C(x_{0},x_{i})}
\end{displaymath}
and
\begin{displaymath}
\mathbb{P}_{x_{0}}\big(B^{\widetilde{\mathcal{G}}}
_{T_{\lbrace x_{1},\dots,x_{\deg(x_{0})}\rbrace}}=x_{j}\big)=
\dfrac{C(x_{0},x_{j})}{\sum_{i=1}^{\deg x_{0}}C(x_{0},x_{i})};
\end{displaymath}
see also Theorem $2.1$ in \cite{Folz2014VolGrowthStochCompl}. This explains our particular choice of the lengths $(\rho(e))_{e\in E}$.

From now on the Brownian motion $B^{\widetilde{\mathcal{G}}}$ on 
$\widetilde{\mathcal{G}}$ is considered to be constructed and the starting point to be arbitrary. It is not excluded that $B^{\widetilde{\mathcal{G}}}$  blows up in finite time. A necessary but not sufficient condition of this is the existence of a path of finite length that visits infinitely many vertices. Let $\tilde{\kappa}$ be the following measure on $\widetilde{\mathcal{G}}$:
\begin{displaymath}
\tilde{\kappa}:=\sum_{x\in V}\kappa(x)\delta_{x}.
\end{displaymath}
Let $\tilde{\zeta}$ be the first time either $B^{\widetilde{\mathcal{G}}}$ blows up or the additive functional
\begin{displaymath}
\int_{y\in\widetilde{\mathcal{G}}}L^{y}_{t}(B^{\widetilde{\mathcal{G}}})
\tilde{\kappa}(dy)=
\sum_{x\in V}L^{x}_{t}(B^{\widetilde{\mathcal{G}}})\kappa(x)
\end{displaymath}
hits an independent exponential time with mean $1$. $\tilde{\zeta}=+\infty$ a.s. if $\kappa\equiv 0$ and $B^{\widetilde{\mathcal{G}}}$ is conservative. For $\ell\geq 0$ let $\tau_{\ell}$ be the stopping time
\begin{displaymath}
\tau_{\ell}:=\inf\Big\lbrace T\geq 0\Big\vert\sum_{x\in V}L^{x}_{t}(B^{\widetilde{\mathcal{G}}})\geq \ell\Big\rbrace.
\end{displaymath}
If the starting point of $B^{\widetilde{\mathcal{G}}}$ is a vertex then
the process $(B^{\widetilde{\mathcal{G}}}_{\tau_{l}})_
{0\leq \ell<\sum_{x\in V}L^{x}_{\tilde{\zeta}}(B^{\widetilde{\mathcal{G}}})}$ has the same law as the Markov jump process $X$ on $V$. In particular it follows that the process
$(B^{\widetilde{\mathcal{G}}}_{t})_{0\leq t<\tilde{\zeta}}$ is transient.

The $0$-potential of the process $(B^{\widetilde{\mathcal{G}}}_{t})_{0\leq t<\tilde{\zeta}}$ has a density
relative to the measure $m$, the Green's function 
$(G(y,z))_{y,z\in \widetilde{\mathcal{G}}}$. We use the same notation as for the Green's function of $X$ because the latter is the restriction to $V$ of the first. The value of $(G(y,z))_{y,z\in \widetilde{\mathcal{G}}}$ on the interior of the edges is obtained from its value on the vertices by interpolation. Let $(x_{1},y_{1})$ and
$(x_{2},y_{2})$ be two pairs of adjacent vertices in $\mathcal{G}$. Let $z_{1}$ respectively $z_{2}$ be a point in the interval $[x_{1},y_{1}]$
respectively $[x_{2},y_{2}]$ and $r_{1}$ respectively $r_{2}$ be the length of $[x_{1},z_{1}]$ respectively $[x_{2},z_{2}]$. Then
\begin{multline}
\label{SecCouplComplex: EqLinInterpol}
G(z_{1},z_{2})=\dfrac{1}{\rho(\lbrace x_{1},y_{1}\rbrace)
\rho(\lbrace x_{2},y_{2}\rbrace)}
\Big((\rho(\lbrace x_{1},y_{1}\rbrace)-r_{1})
(\rho(\lbrace x_{2},y_{2}\rbrace)-r_{2})G(x_{1},x_{2})\\
+r_{1}r_{2}G(y_{1},y_{2})+
r_{1}(\rho(\lbrace x_{2},y_{2}\rbrace)-r_{2})G(y_{1},x_{2})+
(\rho(\lbrace x_{1},y_{1}\rbrace)-r_{1})r_{2}G(x_{1},y_{2})\Big)\\
+1_{\lbrace x_{1},y_{1}\rbrace=\lbrace x_{2},y_{2}\rbrace}2
\left(r_{1}\wedge r_{2}-\dfrac{r_{1}r_{2}}{\rho(\lbrace x_{1},y_{1}\rbrace)}\right),
\end{multline}
where by convention $(x_{1},y_{1})=(x_{2},y_{2})$ if 
$\lbrace x_{1},y_{1}\rbrace=\lbrace x_{2},y_{2}\rbrace$.

Let $(\phi_{y})_{y\in \widetilde{\mathcal{G}}}$ be the Gaussian free field on $\widetilde{\mathcal{G}}$ with covariance function $G$.
It's restriction to $V$ is the Gaussian free field on the graph $\mathcal{G}$, hence the same notation. Conditional on $(\phi_{x})_{x\in V}$, $(\phi_{y})_{y\in \widetilde{\mathcal{G}}}$ is obtained by joining on every edge $e$ the two values of $\phi$ on its endpoints by an independent bridge of length $\rho(e)$ of a Brownian motion with variance $2$ at time $1$ (not a standard Brownian bridge).
In particular $(\phi_{y})_{y\in \widetilde{\mathcal{G}}}$ has a continuous version.

The process $(B^{\widetilde{\mathcal{G}}}_{t})_{0\leq t<\tilde{\zeta}}$ fits into the framework of \cite{FitzsimmonsRosen2012LoopsIsomorphism} and one can associate to it a measure on time-parametrized continuous loops $\tilde{\mu}$. Let $\widetilde{\mathcal{L}}_{\alpha}$ be the
Poisson point process of loops of intensity $\alpha\tilde{\mu}$. We would like to stress that by loop we only mean a continuous paths with the same starting and endpoint without assumptions on its homotopy class and actually most loops in $\widetilde{\mathcal{L}}_{\alpha}$ are topologically trivial. Just as the process $(B^{\widetilde{\mathcal{G}}}_{t})_{0\leq t<\tilde{\zeta}}$ itself, the loop $\tilde{\gamma}\in\widetilde{\mathcal{L}}_{\alpha}$ can be endowed with space-time continuous local times $(L^{y}_{t}(\tilde{\gamma}))_{0\leq t\leq T(\tilde{\gamma}),
y\in\widetilde{\mathcal{G}}}$ relative to the measure $m$. The occupation field 
$(\widehat{\mathcal{L}}_{\alpha}^{y})_{y\in\widetilde{\mathcal{G}}}$ is defined as
\begin{displaymath}
\widehat{\mathcal{L}}_{\alpha}^{y}=
\sum_{\tilde{\gamma}\in\widetilde{\mathcal{L}}_{\alpha}}
L^{y}_{T(\tilde{\gamma})}(\tilde{\gamma}).
\end{displaymath}
The restriction of $(\widehat{\mathcal{L}}_{\alpha}^{y})_{y\in\widetilde{\mathcal{G}}}$
to the set of vertices $V$ has the same law as the occupation field of the discrete loops $\mathcal{L}_{\alpha}$, hence the same notation.
As in the discrete case, at $\alpha=\frac{1}{2}$, $(\widehat{\mathcal{L}}_{1/2}^{y})_{y\in\widetilde{\mathcal{G}}}$ has the same law as 
$(\frac{1}{2}\phi^{2}_{y})_{y\in \widetilde{\mathcal{G}}}$ (see Theorem $3.1$ in \cite{FitzsimmonsRosen2012LoopsIsomorphism}).

The discrete-space loops of $\mathcal{L}_{\alpha}$ can be obtained from the continuous loops $\widetilde{\mathcal{L}}_{\alpha}$ by taking the print of the latter on $V$. This is described in 
\cite{FitzsimmonsRosen2012LoopsIsomorphism}, Section $7.3$, or in a less general situation of the restriction of the loops of one-dimensional diffusions to a discrete subset in \cite{Lupu20131dimLoops}, Section $3.7$. We explain how the restriction  from $\widetilde{\mathcal{G}}$
to $V$ works. First of all we consider only the subset
$\lbrace \tilde{\gamma}\in\widetilde{\mathcal{L}}_{\alpha}\vert \tilde{\gamma}~\text{visits}~V\rbrace$ because the print of other loops on $V$ is empty. Next we re-root the loops so as to have the starting point in $V$: to each loop $\tilde{\gamma}$ visiting $V$ we associate a uniform r.v. on $(0,1)$ $U_{\tilde{\gamma}}$, these different r.v.'s being independent conditional on the loops. We introduce the time
\begin{displaymath}
\tau^{V}(\tilde{\gamma}):=\inf\Big\lbrace t\in [0, T(\tilde{\gamma})]
\Big\vert\sum_{x\in V} L^{x}_{t}(\tilde{\gamma})\geq
U_{\tilde{\gamma}}\sum_{x\in V}L^{x}_{T(\tilde{\gamma})}(\tilde{\gamma})\Big\rbrace.
\end{displaymath}
For each loop $\tilde{\gamma}$ visiting $V$ we make a rotation of parametrization so as to have the starting and end-time at $\tau^{V}(\tilde{\gamma})$ instead of $0$. Let $\widetilde{\mathcal{L}}'$
be the set of the new re-parametrized loops. For each 
$\tilde{\gamma}'\in\widetilde{\mathcal{L}}'$ and 
$\ell\in\Big[0,\sum_{x\in V}L^{x}_{T(\tilde{\gamma}')}(\tilde{\gamma}')\Big]$ we define
\begin{displaymath}
\tau^{V}_{\ell}(\tilde{\gamma}'):=
\inf\Big\lbrace t\in [0, T(\tilde{\gamma}')]
\Big\vert\sum_{x\in V} L^{x}_{t}(\tilde{\gamma}')\geq \ell\Big\rbrace.
\end{displaymath}
The set of $V$-valued loops
\begin{displaymath}
\left\lbrace (\tilde{\gamma}'
_{\tau^{V}_{\ell}(\tilde{\gamma}')})_{0\leq \ell\leq\sum_{x\in V}L^{x}
_{T(\tilde{\gamma}')}(\tilde{\gamma}')}\vert 
\tilde{\gamma}'\in\widetilde{\mathcal{L}}'\right\rbrace
\end{displaymath}
has the same law as $\mathcal{L}_{\alpha}$.

Next we explain how to reconstruct 
$\lbrace \tilde{\gamma}\in\widetilde{\mathcal{L}}_{\alpha}\vert \tilde{\gamma}~\text{visits}~V\rbrace$ from $\mathcal{L}_{\alpha}$ by adding random excursions to the discrete-space loops. We won't give the proof of this. For elements supporting what we explain see \cite{LeJan2011Loops}, Chapter $7$, and \cite{Lupu20131dimLoops}, Corollary $3.12$. Let $x_{0}\in V$ and 
$\lbrace x_{1},\dots,x_{\deg(x_{0})}\rbrace$ the vertices adjacent to $x_{0}$. Let $\eta_{+}$ be the intensity measure of positive Brownian excursions. To every loop $\gamma\in\mathcal{L}_{\alpha}$ spending a time $\ell$ in $x_{0}$ before jumping to one of its neighbours or before stopping one has to add excursions from $x_{0}$ to $x_{0}$ in
$\bigcup_{i=1}^{\deg(x_{0})}I_{\lbrace x_{0},x_{i}\rbrace}$ according to a Poisson point process, the intensity of excursions that take place inside the edge $I_{\lbrace x_{0},x_{i}\rbrace}$ being
\begin{equation}
\label{SecCouplComplex: EqIntensity}
\ell\times 1_{\text{height excursion}~<
\rho(\lbrace x_{0},x_{i}\rbrace)} \eta_{+}.
\end{equation}
Let $x,y$ be two adjacent vertices. Whenever a loop $\gamma\in\mathcal{L}_{\alpha}$ jumps from $x$ to $y$ one has to add a Brownian excursion from $x$ to $y$ inside $I_{\lbrace x,y\rbrace}$ (a Brownian excursion from $0$ to $a>0$ is a Bessel-$3$ process started from $0$ run until hitting $a$). All the added excursions have to be independent conditional on $\mathcal{L}_{\alpha}$. At this stage we get a Poisson point process of continuous loops in $\widetilde{\mathcal{G}}$, but all have a starting point lying in $V$. The final step is to choose for each a new random starting point distributed uniformly on their duration. What we get has the law of
$\lbrace \tilde{\gamma}\in\widetilde{\mathcal{L}}_{\alpha}\vert \tilde{\gamma}~\text{visits}~V\rbrace$.

From now on we assume that $\mathcal{L}_{\alpha}$ and $\widetilde{\mathcal{L}}_{\alpha}$ are naturally coupled on the same probability space through restriction. $\widetilde{\mathcal{L}}_{\alpha}$ has loop clusters and we will denote by $\widetilde{\mathfrak{C}}_{\alpha}$ the set of these clusters. Obviously each cluster of $\mathcal{L}_{\alpha}$ is contained in a cluster of $\widetilde{\mathcal{L}}_{\alpha}$, but with positive probability a cluster of $\widetilde{\mathcal{L}}_{\alpha}$ may contain several clusters of $\mathcal{L}_{\alpha}$. We will prove the following:

\begin{proposition}
\label{PropCouplingCont}
There is a coupling between the Poisson ensemble of loops $\widetilde{\mathcal{L}}_{1/2}$ and a continuous version of the Gaussian free field $(\phi_{y})_{y\in\widetilde{\mathcal{G}}}$ such that the two constraints hold:
\begin{itemize}
\item For all $y\in \widetilde{\mathcal{G}}$, 
$\widehat{\mathcal{L}}^{y}_{1/2}=\frac{1}{2}\phi_{y}^{2}$;
\item The clusters of loops of $\widetilde{\mathcal{L}}_{1/2}$ are exactly the sign clusters of $(\phi_{y})_{y\in\widetilde{\mathcal{G}}}$.
\end{itemize}
\end{proposition}

Theorem \ref{ThmCoupling} follows from the above proposition because the restriction of $(\widehat{\mathcal{L}}^{y}_{1/2})_{y\in \widetilde{\mathcal{G}}}$ to $V$ is 
$(\widehat{\mathcal{L}}^{x}_{1/2})_{x\in V}$, the restriction of
$(\phi_{y})_{y\in\widetilde{\mathcal{G}}}$ to $V$ is the Gaussian free field on $\mathcal{G}$ and the sign of $\phi$ is constant on the clusters of $\widetilde{\mathcal{L}}_{1/2}$, hence also constant on the clusters of $\mathcal{L}_{1/2}$. 

The first step in proving Proposition \ref{PropCouplingCont} is to show that there is a realisation of $\widetilde{\mathcal{L}}_{\alpha}$ such that its occupation field 
$(\widehat{\mathcal{L}}^{y}_{\alpha})_{y\in \widetilde{\mathcal{G}}}$ is continuous. We know already that each
individual loop in $\widetilde{\mathcal{L}}_{\alpha}$ has space-time continuous local times and that the process $(\widehat{\mathcal{L}}^{y}_{\alpha})_{y\in \widetilde{\mathcal{G}}}$ considered for itself, regardless of the loops, has a continuous version (see \cite{Lupu20131dimLoops}, Section $4.2$). However this does not automatically imply that a realisation of 
$(\widehat{\mathcal{L}}^{y}_{\alpha})_
{y\in \widetilde{\mathcal{G}}}$ as the occupation field of 
$\widetilde{\mathcal{L}}_{\alpha}$ can be made continuous (there are infinitely many loops above each point in $\widetilde{\mathcal{G}}$ and the occupation field is an infinite sum of continuous functions). A counterexample is given in \cite{Lupu20131dimLoops}, Section $4.2$, the remark after Proposition $4.6$.

\begin{lemma}
\label{LemContOccupField}
There is a realisation of $\widetilde{\mathcal{L}}_{\alpha}$ such that its occupation field 
$(\widehat{\mathcal{L}}^{y}_{\alpha})_{y\in \widetilde{\mathcal{G}}}$ is continuous.
\end{lemma}

\begin{proof}
We divide the loops of $\widetilde{\mathcal{L}}_{\alpha}$ into three classes:
\begin{itemize}
\item[(i)] the loops that visit at least two vertices in $V$;
\item[(ii)] the loops that visit only one vertex in $V$;
\item[(iii)] the loops that do not visit any vertex and are contained in the interior of an edge.
\end{itemize}

Above any vertex $x\in V$ are only finitely many loops of type (i) (see \cite{LeJan2011Loops} Chapter $2$ for the exact expression of their intensity). Each individual loop of type (i) has a continuous occupation field and the sum of this occupation fields is locally finite and therefore continuous.

Let $x_{0}\in V$ and 
$\lbrace x_{1},\dots,x_{\deg(x_{0})}\rbrace$ the vertices adjacent to $x_{0}$. We consider now the loops of type (ii) such that $x_{0}$ is the only vertex they visit, which we denote $(\tilde{\gamma}_{j})_{j\geq 0}$. Conditional on $L^{x_{0}}_{T(\tilde{\gamma}_{j})}(\tilde{\gamma}_{j})$, $\tilde{\gamma}_{j}$ is obtained by launching excursion from $x_{0}$ to $x_{0}$ in $\bigcup_{i=1}^{\deg(x_{0})}I_{\lbrace x_{0},x_{i}\rbrace}$ according to a Poisson point process, the intensity of excursions that take place inside the edge 
$I_{\lbrace x_{0},x_{i}\rbrace}$ being 
[see \eqref{SecCouplComplex: EqIntensity}]
\begin{displaymath}
L^{x_{0}}_{T(\tilde{\gamma}_{j})}(\tilde{\gamma}_{j})
\times 1_{\text{height excursion}~<\rho(\lbrace x_{0},x_{i}\rbrace)} \eta_{+}.
\end{displaymath}
If we consider all the loops $(\tilde{\gamma}_{j})_{j\geq 0}$ we obtain an intensity
\begin{displaymath}
\Big(\sum_{j\geq 0}L^{x_{0}}_{T(\tilde{\gamma}_{j})}(\tilde{\gamma}_{j})\Big)
\times 1_{\text{height excursion}~<\rho(\lbrace x_{0},x_{i}\rbrace)} \eta_{+}.
\end{displaymath}
The continuity of the occupation field of 
$(\tilde{\gamma}_{j})_{j\geq 0}$ follows from the continuity of Brownian local times.

Let $e$ be an edge. We consider the loops of type (iii) that are contained in $I_{e}$. They have the same law as a Poisson ensemble of loops of parameter $\alpha$ associated to the standard Brownian motion on the bounded interval $I_{e}$ killed upon reaching either of its boundary points. This situation was entirely covered in \cite{Lupu20131dimLoops}. According to Corollary $5.5$ in \cite{Lupu20131dimLoops} it is possible to construct these loops and a continuous version on their occupation field on the same probability space. All the subtlety of our lemma lies in this point. Moreover according to Proposition $4.7$ in \cite{Lupu20131dimLoops} the occupation field of these loops converges to $0$ at the end-vertices of $I_{e}$.
\end{proof}

From now on we consider only the continuous realization of the occupation field 
$(\widehat{\mathcal{L}}^{y}_{\alpha})_{y\in \widetilde{\mathcal{G}}}$. We call a positive component of $(\widehat{\mathcal{L}}^{y}_{\alpha})_{y\in \widetilde{\mathcal{G}}}$ a maximal connected subset of $\widetilde{\mathcal{G}}$ on which the occupation field is positive. It is open and by continuity the occupation field is zero on the boundary of a positive component. Given a continuous loop $\tilde{\gamma}$, 
$\operatorname{Range}(\tilde{\gamma})$ will denote its range.

\begin{lemma}
\label{LemClustPosComp}
Let $\widetilde{\mathcal{C}}\in\widetilde{\mathfrak{C}}_{\alpha}$ be a cluster of 
$\widetilde{\mathcal{L}}_{\alpha}$. Then
\begin{displaymath}
\bigcup_{\tilde{\gamma}\in\widetilde{\mathcal{C}}}
\operatorname{Range}(\tilde{\gamma})
\end{displaymath}
is a positive component of $(\widehat{\mathcal{L}}^{y}_{\alpha})_{y\in \widetilde{\mathcal{G}}}$. Conversely every positive component of $(\widehat{\mathcal{L}}^{y}_{\alpha})_{y\in \widetilde{\mathcal{G}}}$ is of this form.
\end{lemma}

\begin{proof}
The following almost sure properties hold:
\begin{itemize}
\item[(i)] for every $\tilde{\gamma}\in\widetilde{\mathcal{L}}_{\alpha}$
the occupation field of $\tilde{\gamma}$ is positive in the interior of
$\operatorname{Range}(\tilde{\gamma})$ and zero on the boundary 
$\partial \operatorname{Range}(\tilde{\gamma})$;
\item[(ii)] for every $\tilde{\gamma}\in\widetilde{\mathcal{L}}_{\alpha}$ and $y\in\partial \operatorname{Range}(\tilde{\gamma})$, there is another loop $\tilde{\gamma}'\in\widetilde{\mathcal{L}}_{\alpha}$ such that $y$ is contained in the interior of 
$\operatorname{Range}(\tilde{\gamma}')$.
\end{itemize}

Property (i) comes from an analogous property of a finite duration one-dimensional Brownian path: its occupation field is positive on its range, except at the maximum and the minimum where it is zero.

We briefly explain why property (ii) is true. First of all the boundary $\partial \operatorname{Range}(\tilde{\gamma})$ is finite because it can intersect an edge in at most two points and a loop visits finitely many edges. Moreover any deterministic point in $\widetilde{\mathcal{G}}$ is almost surely covered by the interior of the range of a loop. Applying Palm's identity one gets (ii):
\begin{multline*}
\mathbb{E}\Big[\sharp\Big\lbrace\tilde{\gamma}\in\widetilde{\mathcal{L}}_{\alpha}\Big\vert
\partial \operatorname{Range}(\tilde{\gamma})
\not\subseteq 
\bigcup_{\tilde{\gamma}'\in \widetilde{\mathcal{L}}_{\alpha}}
\operatorname{Int}(\operatorname{Range}(\tilde{\gamma}'))
\Big\rbrace\Big]\\=
\alpha\int\mathbb{P}
\Big(
\partial \operatorname{Range}(\tilde{\gamma})
\not\subseteq 
\bigcup_{\tilde{\gamma}'\in \widetilde{\mathcal{L}}_{\alpha}}
\operatorname{Int}(\operatorname{Range}(\tilde{\gamma}'))
\Big)
\tilde{\mu}(d\tilde{\gamma})
=0.
\end{multline*}

Properties (i) and (ii) imply on one hand that the zero set of $(\widehat{\mathcal{L}}^{y}_{\alpha})_{y\in \widetilde{\mathcal{G}}}$ is exactly the set of all point in $\widetilde{\mathcal{G}}$ that are not visited by any loop in $\widetilde{\mathcal{L}}_{\alpha}$ and on the other hand that any point visited by a loop cannot belong to the boundary of a cluster of loops. This in turn implies the lemma.
\end{proof}

\begin{proof}[Proof of Proposition \ref{PropCouplingCont}]
First sample $\widetilde{\mathcal{L}}_{1/2}$ with a continuous version of its occupation field. Consider 
$\Big(\sqrt{2\widehat{\mathcal{L}}^{y}_{1/2}}\Big)
_{y\in \widetilde{\mathcal{G}}}$ as a realization of 
$(\vert\phi_{y}\vert)_{y\in \widetilde{\mathcal{G}}}$ and sample the sign of the Gaussian free field $\phi$ independently from 
$\widetilde{\mathcal{L}}_{1/2}$ conditional on
$\Big(\sqrt{2\widehat{\mathcal{L}}^{y}_{1/2}}\Big)
_{y\in \widetilde{\mathcal{G}}}$. 
Then according to Lemma \ref{LemClustPosComp} the clusters of 
$\widetilde{\mathcal{L}}_{1/2}$ are exactly the positive components of 
$(\vert\phi_{y}\vert)_{y\in \widetilde{\mathcal{G}}}$ 
which are the sign clusters of
$(\phi_{y})_{y\in \widetilde{\mathcal{G}}}$.
\end{proof}

\section{Alternative description of the coupling}
\label{SecAltDesc}

In this section we give en alternative description on the coupling between $\mathcal{L}_{1/2}$ and $(\phi_{x})_{x\in V}$ constructed in Section \ref{SecCouplComplex} but that does not use
$\widetilde{\mathcal{L}}_{1/2}$ as intermediate. First we deal with the law of the sign of $\phi$ conditional on $(\vert\phi_{y}\vert)_{y\in\widetilde{\mathcal{G}}}$. We will show that one has to chose the sign independently for each positive component of 
$(\vert\phi_{y}\vert)_{y\in\widetilde{\mathcal{G}}}$ and uniformly distributed in $\lbrace -1,+1\rbrace$. Then we will deal with the probability of a cluster of continuous loops occupying entirely an edge $e$ conditional on discrete-space loops $\mathcal{L}_{1/2}$ and on the event that none of these loops occupies $e$.

Let $K$ be a non-empty compact connected subset of $\widetilde{\mathcal{G}}$. $\partial K$ is finite, $\widetilde{\mathcal{G}}\setminus K$ has finitely many connected components and the closure of each of these connected components is itself a metric graph associated to some discrete graph. Let $T_{K}$ be the first time the Brownian motion $B^{\widetilde{\mathcal{G}}}$, started outside $K$, hits $K$. Let $(G^{\widetilde{\mathcal{G}}\setminus K}(y,z))
_{y,z\in\widetilde{\mathcal{G}}\setminus K}$ be the Green's function relative to the measure $m$ of the killed process $(B^{\widetilde{\mathcal{G}}}_{t})_{0\leq t<\tilde{\zeta}\wedge T_{K}}$.
$G^{\widetilde{\mathcal{G}}\setminus K}$ is symmetric, continuous and extends continuously to 
$\overline{\widetilde{\mathcal{G}}\setminus K}$ by taking value $0$ on the boundary. Actually $G^{\widetilde{\mathcal{G}}\setminus K}$ is obtained by interpolation from its values on the vertices and $\partial K$ as in \eqref{SecCouplComplex: EqLinInterpol}. Let $(\phi^{\widetilde{\mathcal{G}}\setminus K}_{y})_{y\in \widetilde{\mathcal{G}}\setminus K}$ be the Gaussian free field on
$\widetilde{\mathcal{G}}\setminus K$ with variance-covariance function
$G^{\widetilde{\mathcal{G}}\setminus K}$. Let $f$ be a function on $\partial K$ and $u_{f,K}$ be the following function on $\widetilde{\mathcal{G}}\setminus K$:
\begin{displaymath}
u_{f,K}(y):=\mathbb{E}_{y}\left[f(B^{\widetilde{\mathcal{G}}}_{T_{K}})
1_{T_{K}<\tilde{\zeta}}\right].
\end{displaymath}
By the Markov property of $(\phi_{y})_{y\in\widetilde{\mathcal{G}}}$, 
conditional on $(\phi_{y})_{y\in K}$,
$(\phi_{y})_{y\in\widetilde{\mathcal{G}}\setminus K}$ has the same law
as $(u_{\phi,K}(y)+\phi^{\widetilde{\mathcal{G}}\setminus K}
_{y})_{y\in\widetilde{\mathcal{G}}\setminus K}$. We consider now a random connected compact subset $\mathcal{K}$. We use the equivalent $\sigma$-algebras on the connected compact subsets:
\begin{itemize}
\item the $\sigma$-algebra induced by the events
$(\lbrace \mathcal{K}\subseteq U\rbrace)_{U~\text{open subset of}~\widetilde{\mathcal{G}}}$,
\item the $\sigma$-algebra induced by the events
$(\lbrace F\cap\mathcal{K}\neq\emptyset\rbrace)
_{F~\text{closed subset of}~\widetilde{\mathcal{G}}}$.
\end{itemize}
Below we state a strong Markov property for the Gaussian free field $(\phi_{y})_{y\in\widetilde{\mathcal{G}}}$. It can be derived from the simple Markov property; see \cite{Rozanov1982MarkovRandomFields}, Chapter $2$, Section $2.4$, Theorem $4$.

\begin{strongMarkov}
Let $\mathcal{K}$ be a random compact connected subset of
$\widetilde{\mathcal{G}}$ such that for every deterministic open
subset $U$ of $\widetilde{\mathcal{G}}$ the event 
$\lbrace\mathcal{K}\subseteq U\rbrace$ is measurable with respect to
$(\phi_{y})_{y\in U}$. Then conditional on $\mathcal{K}$ and 
$(\phi_{y})_{y\in \mathcal{K}}$, $(\phi_{y})_{y\in\widetilde{\mathcal{G}}\setminus \mathcal{K}}$ has the same law as $(u_{\phi,\mathcal{K}}(y)+
\phi^{\widetilde{\mathcal{G}}\setminus \mathcal{K}}
_{y})_{y\in\widetilde{\mathcal{G}}\setminus \mathcal{K}}$.
\end{strongMarkov}

\begin{lemma}
\label{LemChangeSignComponent}
Given $y_{0}\in\widetilde{\mathcal{G}}$ we denote by $F_{y_{0}}$ the closure of the positive component of $(\vert\phi_{y}\vert)_{y\in\widetilde{\mathcal{G}}}$ containing $y_{0}$
(a.s. $\phi_{y_{0}}\neq 0$). Then the field 
$(-1_{y\in F_{y_{0}}}\phi_{y}+1_{y\not\in F_{y_{0}}}\phi_{y})
_{y\in\widetilde{\mathcal{G}}}$ has the same law as the Gaussian free field $(\phi_{y})_{y\in\widetilde{\mathcal{G}}}$.
\end{lemma}

\begin{proof}
By construction $F_{y_{0}}$ is closed and connected, but not necessarily compact if $V$ is not finite. $\phi$ is zero on $\partial F_{y_{0}}$. 

We first consider the case of $V$ being finite. Then $F_{y_{0}}$ is compact. According to the strong Markov property, conditional on
$F_{y_{0}}$ and $(\phi_{y})_{y\in F_{y_{0}}}$, $(\phi_{y})_{y\in\widetilde{\mathcal{G}}\setminus F_{y_{0}}}$ has the same law as $(\phi^{\widetilde{\mathcal{G}}\setminus F_{y_{0}}}
_{y})_{y\in\widetilde{\mathcal{G}}\setminus F_{y_{0}}}$. But
$\phi^{\widetilde{\mathcal{G}}\setminus F_{y_{0}}}$ and
$-\phi^{\widetilde{\mathcal{G}}\setminus F_{y_{0}}}$ have the same law. Thus $(1_{y\in F_{y_{0}}}\phi_{y}-1_{y\not\in F_{y_{0}}}\phi_{y})
_{y\in\widetilde{\mathcal{G}}}$ has the same law as $\phi$. Since $\phi$ and $-\phi$ have the same law, $(-1_{y\in F_{y_{0}}}\phi_{y}+1_{y\not\in F_{y_{0}}}\phi_{y})_{y\in\widetilde{\mathcal{G}}}$ has the same law as $\phi$ too.

If $V$ is infinite, let $x_{0}\in V$. Let $V_{n}$ be the set of vertices separated from $x_{0}$ by at most $n$ edges. $V_{n}$ is finite.
$V_{0}=\lbrace x_{0}\rbrace$ and $V_{1}$ is made of $x_{0}$ and all its neighbours. For $n\geq 1$ let $E_{n}$ be the set of edges either connecting two vertices in $V_{n-1}$ or a vertex in 
$V_{n}\setminus V_{n-1}$ to a vertex in $V_{n-1}$. 
$\mathcal{G}_{n}:=(V_{n},E_{n})$ is a connected sub-graph of $\mathcal{G}$. Let $\widetilde{\mathcal{G}}_{n}$ be the metric graph associated to the graph $\mathcal{G}_{n}$, viewed as a compact subset of $\widetilde{\mathcal{G}}$. For $n$ large enough such that $y_{0}\in\widetilde{\mathcal{G}}_{n}$, let 
$F_{y_{0},n}$ be the positive component of 
$(\vert\phi^{\widetilde{\mathcal{G}}\setminus(V_{n}\setminus V_{n-1})}_{y}\vert)_{y\in\widetilde{\mathcal{G}}}$ containing $y_{0}$, which is compact. As in the previous case, $(-1_{y\in F_{y_{0},n}}\phi^{\widetilde{\mathcal{G}}
\setminus(V_{n}\setminus V_{n-1})}_{y}+1_{y\not\in F_{y_{0},n}}\phi^{\widetilde{\mathcal{G}}\setminus(V_{n}\setminus V_{n-1})}_{y})_{y\in\widetilde{\mathcal{G}}}$ has the same law
as $\phi^{\widetilde{\mathcal{G}}\setminus(V_{n}\setminus V_{n-1})}$. 
As $n$ converges to $+\infty$, the first field converges in law to
$(-1_{y\in F_{y_{0}}}\phi_{y}+1_{y\not\in F_{y_{0}}}
\phi_{y})_{y\in\widetilde{\mathcal{G}}}$ and the second field converges in law to $\phi$, which proves the lemma.
\end{proof}

\begin{lemma}
\label{LemSignGFF}
Conditional on $(\vert\phi_{y}\vert)_{y\in\widetilde{\mathcal{G}}}$, the sign of $\phi$ on each of its connected components is distributed independently and uniformly in $\lbrace -1,+1\rbrace$.
\end{lemma}

\begin{proof}
Let $(y_{n})_{n\geq 0}$ be a dense sequence in $\widetilde{\mathcal{G}}$. Let $(\sigma_{n})_{n\geq 0}$ be an i.i.d. sequence of uniformly distributed variables in $\lbrace -1,+1\rbrace$ independent of $\phi$. According to lemma \ref{LemChangeSignComponent}, the field
\begin{displaymath}
\left(\prod_{n=0}^{N}(\sigma_{n}1_{y\in F_{y_{n}}}
+1_{y\not\in F_{y_{n}}})
\times\phi_{y}\right)_{y\in\widetilde{\mathcal{G}}}
\end{displaymath}
has the same law as $\phi$ whatever the value of $N$. Moreover as $N$ converges to $+\infty$, this field converges in law to the field obtained by choosing uniformly and independently a sign for each positive component of $(\vert\phi_{y}\vert)_{y\in\widetilde{\mathcal{G}}}$.
\end{proof}

Next we consider the discrete-space loops 
$\mathcal{L}_{1/2}$ and continuous loops 
$\widetilde{\mathcal{L}}_{1/2}$ coupled in the natural way though the restriction of the latter to $V$. We deal with the probability of a cluster of continuous loops occupying entirely an edge $e$ conditional on $\mathcal{L}_{1/2}$ and on the event that none of discrete-space loops occupies $e$. This event is the same as the occupation field $\widehat{\mathcal{L}}_{1/2}$ staying positive on $I_{e}$ and not having zeros there. Let $e=\lbrace x,y\rbrace$ be an edge joining vertices $x$ and $y$. In case $e$ is not occupied by a loop of 
$\mathcal{L}_{1/2}$, there are three kind of paths visiting $I_{e}$:
\begin{itemize}
\item The loops of entirely $\widetilde{\mathcal{L}}_{1/2}$ contained in $I_{e}$. These are independent $\mathcal{L}_{1/2}$ as they have no print on $V$. The occupation field of these loops is the square of a standard Brownian bridge of length $\rho(e)$ from $0$ at $x$ to $0$ at $y$ (\cite{Lupu20131dimLoops}, Proposition $4.6$).
\item The Poisson point process of excursions from $x$ to $x$ inside $I_{e}$ of the loops in $\widetilde{\mathcal{L}}_{1/2}$ visiting $x$. The intensity of excursions is
\begin{displaymath}
\widehat{\mathcal{L}}^{x}_{1/2}
\times 1_{\text{height excursion}~<\rho(e)} \eta_{+}.
\end{displaymath}
Conditional on $\widehat{\mathcal{L}}^{x}_{1/2}$, this Poisson
point process of excursions is independent from $\mathcal{L}_{1/2}$. 
Its occupation field is according to the
second Ray-Knight theorem the square of a Bessel-$0$ process with initial value $\widehat{\mathcal{L}}^{x}_{1/2}$ at $x$ conditioned to hit $0$ before time $\rho(e)$.
\item The Poisson point process of excursions from $y$ to $y$ inside $I_{e}$ of the loops in $\widetilde{\mathcal{L}}_{1/2}$ visiting $y$. The picture is the same as above.
\end{itemize}
We will denote by $(b^{(T)}_{t})_{0\leq t\leq T}$ a standard Brownian bridge from $0$ to $0$ of length $T$ and 
$(\beta^{(T,\ell)}_{t})_{t\geq 0}$ a square of a Bessel $0$ process starting from $\ell$ at $t=0$ and conditioned to hit $0$ before time $T$.
We have the following picture:

\begin{property}
\label{PropertyIndep}
Conditional on the discrete-space loops $\mathcal{L}_{1/2}$, the events of the family $\Big(\lbrace\widehat{\mathcal{L}}_{1/2}
~\text{has a zero on}~I_{e}\rbrace
\Big)_{e\in E\setminus \bigcup_{\mathcal{C}\in\mathfrak{C}
_{1/2}}\mathcal{C}}$ are independent. Let $e=\lbrace x,y\rbrace$  be an edge. The probability
\begin{displaymath}
\mathbb{P}\bigg(\widehat{\mathcal{L}}_{1/2}~\text{has a zero on}~
I_{e}\Big\vert \mathcal{L}_{1/2},
e\in E\setminus \bigcup_{\mathcal{C}\in
\mathfrak{C}_{1/2}}\mathcal{C}\bigg)
\end{displaymath}
is the same as for the sum of three independent processes
\begin{displaymath}
\Big(b^{(\rho(e)) 2}_{t}+
\beta^{(\rho(e),\widehat{\mathcal{L}}^{x}_{1/2})}_{t}+
\beta^{(\rho(e),\widehat{\mathcal{L}}^{y}_{1/2})}_{\rho(e)-t}
\Big)_{0\leq t\leq\rho(e)}
\end{displaymath}
having a zero on $(0,\rho(e))$.
\end{property}

\begin{lemma}
\label{LemProbInt}
Let $T,\ell_{1},\ell_{2}>0$. The probability that the sum of three independent processes
\begin{equation}
\label{SecAltDesc: EqThreeProc}
\Big(b^{(T) 2}_{t}+\beta^{(T,\ell_{1})}_{t}+
\beta^{(T,\ell_{2})}_{T-t}\Big)_{0\leq t\leq T}
\end{equation}
has a zero on $(0,T)$ is
\begin{equation}
\label{SecAltDesc: BigIntegral}
\dfrac{1}{\sqrt{\pi}}
\int_{0}^{+\infty}\exp\Big(-\dfrac{\ell_{1}\ell_{2}}{(2T)^{2}s}-s\Big)
\dfrac{ds}{\sqrt{s}}.
\end{equation}
\end{lemma}

\begin{proof}
We will break the symmetry of the expression \eqref{SecAltDesc: EqThreeProc} and use the fact that the process
$\Big(b^{(T) 2}_{t}+
\beta^{(T,\ell_{2})}_{T-t}\Big)_{0\leq t\leq T}$ has the same law as the square of a standard Brownian bridge of length $T$ from $0$ to 
$\sqrt{\ell_{2}}$; 
see \cite{RevuzYor1999BMGrundlehren}, Chapter XI, Section $3$.
For the process \eqref{SecAltDesc: EqThreeProc} to have a zero on
$(0,T)$, the process $\beta^{(T,\ell_{1})}$ has to hit $0$ before the last zero of $\Big(b^{(T) 2}_{t}+
\beta^{(T,\ell_{2})}_{T-t}\Big)_{0\leq t\leq T}$.

According to Ray-Knight theorem, the time when the square Bessel $0$ started from $\ell_{1}$ hits $0$ has the same law as the maximum of a standard Brownian motion started from $0$ and stopped at its local time at $0$ reaching the level $\ell_{1}$. The distribution of this maximum is
\begin{displaymath}
1_{a>0}\dfrac{\ell_{1}}{2a^{2}}\exp\left(-\dfrac{\ell_{1}}{2a}\right)da.
\end{displaymath}
In $\beta^{(T,\ell_{1})}$ we condition on hitting zero before time $T$. So the distribution of the first zero is
\begin{equation}
\label{SecAltDesc: Distribt1}
1_{0<t_{2}<T}\dfrac{\ell_{1}}{2t_{1}^{2}}
\exp\left(\dfrac{\ell_{1}}{2T}-\dfrac{\ell_{1}}{2t_{1}}\right)dt_{1}.
\end{equation}

Let $(B_{t})_{t\geq 0}$ be a standard Brownian motion on $\mathbb{R}$ started from $0$ and
\begin{displaymath}
g_{T}:=\sup\lbrace t\in[0,T]\vert B_{t}=0\rbrace.
\end{displaymath}
The joint distribution of $(g_{T},B_{T})$ is (see \cite{RevuzYor1999BMGrundlehren}, Chapter XII, Section $3$)
\begin{displaymath}
1_{0<a<T}\dfrac{\vert x\vert\exp\left(-\dfrac{x^{2}}{2(T-a)}\right)}
{2\pi\sqrt{a(T-a)^{3}}} da dx.
\end{displaymath}
If we condition by $B_{T}=\sqrt{\ell_{2}}$ we get the distribution of the last zero of $\Big(b^{(T) 2}_{t}+
\beta^{(T,\ell_{2})}_{T-t}\Big)_{0\leq t\leq T}$ which is
\begin{equation}
\label{SecAltDesc: Distribt2}
1_{0<t_{1}<T}\dfrac{\sqrt{\ell_{2}T}
\exp\left(\dfrac{\ell_{2}}{2T}-\dfrac{\ell_{2}}{2(T-t_{2})}\right)}
{\sqrt{2\pi t_{2}(T-t_{2})^{3}}} dt_{2}.
\end{equation}

Gathering \eqref{SecAltDesc: Distribt1} and \eqref{SecAltDesc: Distribt2} we get that the probability that we are interested in is
\begin{multline*}
\sqrt{\dfrac{\ell_{2}T}{2\pi}}\exp\left(\dfrac{\ell_{1}+\ell_{2}}{2T}\right)
\int_{0<t_{1}<t_{2}<T}\dfrac{\ell_{1}}{2t_{1}^{2}}
\exp\left(-\dfrac{\ell_{1}}{2t_{1}}\right)
\dfrac{\exp\left(-\dfrac{\ell_{2}}{2(T-t_{2})}\right)}
{\sqrt{t_{2}(T-t_{2})^{3}}} dt_{1} dt_{2}\\
=\sqrt{\dfrac{\ell_{2}T}{2\pi}}\exp\left(\dfrac{\ell_{1}+\ell_{2}}{2T}\right)\int_{0<t_{2}<T}
\exp\left(-\dfrac{\ell_{1}}{2t_{2}}-\dfrac{\ell_{2}}{2(T-t_{2})}\right)
\dfrac{dt_{2}}{\sqrt{t_{2}(T-t_{2})^{3}}}.
\end{multline*}
By performing the change of variables
\begin{displaymath}
s:=\dfrac{\ell_{2}}{2T}\dfrac{t_{2}}{T-t_{2}}
\end{displaymath}
we get the integral \eqref{SecAltDesc: BigIntegral}.
\end{proof}

\begin{lemma}
\label{LemIntegral}
For all $\lambda\geq 0$
\begin{displaymath}
\int_{0}^{+\infty}\exp\Big(-\dfrac{\lambda}{s}-s\Big)
\dfrac{ds}{\sqrt{s}}=\sqrt{\pi}e^{-2\sqrt{\lambda}}.
\end{displaymath}
\end{lemma}

\begin{proof}
Let
\begin{displaymath}
f(\lambda):=\int_{0}^{+\infty}\exp\Big(-\dfrac{\lambda}{s}-s\Big)
\dfrac{ds}{\sqrt{s}}.
\end{displaymath}
Then $f(0)=\Gamma(\frac{1}{2})=\sqrt{\pi}$ and
\begin{displaymath}
f'(\lambda)=-\int_{0}^{+\infty}\exp\Big(-\dfrac{\lambda}{s}-s\Big)
\dfrac{ds}{\sqrt{s^{3}}}.
\end{displaymath}
By doing the change of variables $z=\frac{\lambda}{s}$ we get
\begin{displaymath}
f'(\lambda)=-\dfrac{1}{\sqrt{\lambda}}\int_{0}^{+\infty}
\exp\Big(-z-\dfrac{\lambda}{z}\Big)\dfrac{dz}{\sqrt{z}}.
\end{displaymath}
$f$ satisfies the ODE
\begin{displaymath}
f'(\lambda)=-\dfrac{1}{\sqrt{\lambda}}f(\lambda)
\end{displaymath}
with initial condition $f(0)=\sqrt{\pi}$, thus 
$f(\lambda)=\sqrt{\pi}e^{-2\sqrt{\lambda}}$.
\end{proof}

\begin{corollary}
\label{CorProbAbsEdge}
Conditional on the discrete-space loops $\mathcal{L}_{1/2}$, the events of the family $\Big(\lbrace\widehat{\mathcal{L}}_{1/2}
~\text{has a zero on}~I_{e}\rbrace
\Big)_{e\in E\setminus \bigcup_{\mathcal{C}\in
\mathfrak{C}_{1/2}}\mathcal{C}}$ are independent and the corresponding probabilities are given by
\begin{multline*}
\mathbb{P}\bigg(\widehat{\mathcal{L}}_{1/2}~\text{has a zero on}~
I_{\lbrace x,y\rbrace}\Big\vert \mathcal{L}_{1/2},
\lbrace x,y\rbrace\in E\setminus \bigcup_{\mathcal{C}\in\mathfrak{C}
_{1/2}}\mathcal{C}\bigg)
\\=
\exp\Big(-\dfrac{1}{\rho(\lbrace x,y\rbrace)}
\sqrt{\widehat{\mathcal{L}}^{x}_{1/2}
\widehat{\mathcal{L}}^{y}_{1/2}}\Big)
=\exp\Big(-2C(x,y)\sqrt{\widehat{\mathcal{L}}^{x}_{1/2}
\widehat{\mathcal{L}}^{y}_{1/2}}\Big).
\end{multline*}
\end{corollary}

From Lemma \ref{LemSignGFF} and Corollary \ref{CorProbAbsEdge} follows the alternative description of the coupling between
$\mathcal{L}_{1/2}$ and $(\phi_{x})_{x\in V}$ given by Theorem \ref{ThmAltDescCoupl} bis; see Figure \ref{Fig one}.

Observe that a posteriori the quantity $1-\exp\Big(-2C(x,y)\sqrt{\widehat{\mathcal{L}}^{x}_{1/2}
\widehat{\mathcal{L}}^{y}_{1/2}}\Big)$ equals 
\begin{equation}
\label{EqNiceExpr}
1-e^{-C(x,y)\vert\phi_{x}\phi_{y}\vert}.
\end{equation}

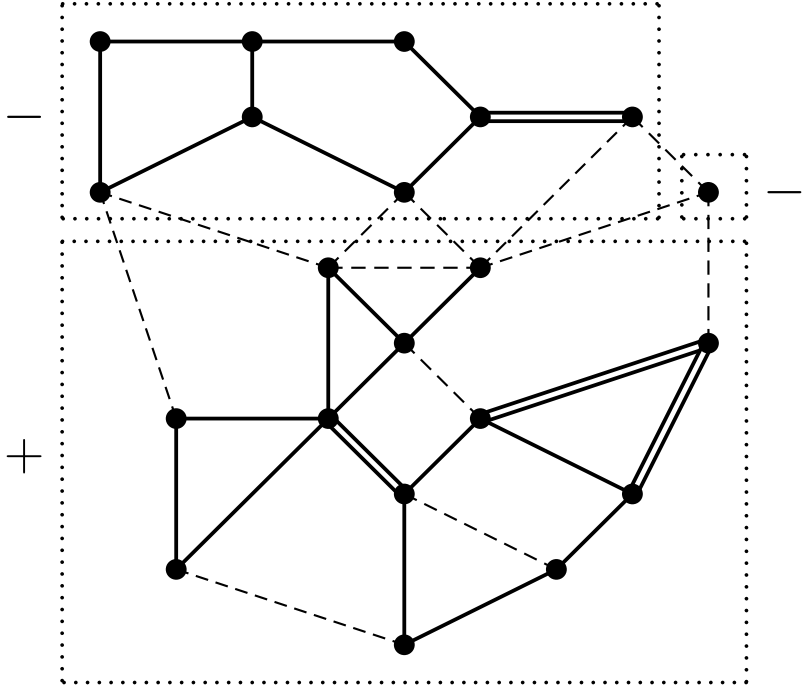
\begin{figure}
\begin{center}
\begin{pspicture}(0,0.5)(10,9.5)
\psline[linewidth=0.05](5,5)(6,6)
\psline[linewidth=0.05](5,5)(4,6)
\psline[linewidth=0.05](5,5)(4,4)
\psline[linestyle=dashed](5,5)(6,4)
\psline[linestyle=dashed](4,6)(6,6)
\psline[linewidth=0.05](4,6)(4,4)
\psline[linewidth=0.05](2,4)(4,4)
\psline[linestyle=dashed](2,4)(1,7)
\psline[linestyle=dashed](4,6)(1,7)
\psline[linestyle=dashed](4,6)(5,7)
\psline[linestyle=dashed](6,6)(5,7)
\psline[linewidth=0.05](3,8)(5,7)
\psline[linewidth=0.05](3,8)(1,7)
\psline[linewidth=0.05](3,9)(3,8)
\psline[linewidth=0.05](3,9)(1,9)
\psline[linewidth=0.05](1,7)(1,9)
\psline[linewidth=0.05](3,9)(5,9)
\psline[linewidth=0.05](6,8)(5,9)
\psline[linewidth=0.05](6,8)(5,7)
\psline[linewidth=0.05, doubleline=true](6,8)(8,8)
\psline[linestyle=dashed](6,6)(8,8)
\psline[linestyle=dashed](6,6)(9,7)
\psline[linestyle=dashed](8,8)(9,7)
\psline[linestyle=dashed](9,5)(9,7)
\psline[linewidth=0.05, doubleline=true](9,5)(6,4)
\psline[linewidth=0.05, doubleline=true](9,5)(8,3)
\psline[linewidth=0.05](8,3)(6,4)
\psline[linewidth=0.05](8,3)(7,2)
\psline[linestyle=dashed](5,3)(7,2)
\psline[linewidth=0.05, doubleline=true](5,3)(4,4)
\psline[linewidth=0.05](5,3)(6,4)
\psline[linewidth=0.05](5,3)(5,1)
\psline[linewidth=0.05](7,2)(5,1)
\psline[linestyle=dashed](2,2)(5,1)
\psline[linewidth=0.05](2,2)(2,4)
\psline[linewidth=0.05](2,2)(4,4)
\psline[linewidth=0.05,linestyle=dotted](5.5,6.65)(0.5,6.65)
\psline[linewidth=0.05,linestyle=dotted](0.5,9.5)(0.5,6.65)
\psline[linewidth=0.05,linestyle=dotted](8.35,9.5)(0.5,9.5)
\psline[linewidth=0.05,linestyle=dotted](8.35,9.5)(8.35,6.65)
\psline[linewidth=0.05,linestyle=dotted](5.5,6.65)(8.35,6.65)
\psline[linewidth=0.05,linestyle=dotted](9.5,6.35)(0.5,6.35)
\psline[linewidth=0.05,linestyle=dotted](9.5,6.35)(9.5,0.5)
\psline[linewidth=0.05,linestyle=dotted](0.5,0.5)(9.5,0.5)
\psline[linewidth=0.05,linestyle=dotted](0.5,0.5)(0.5,6.35)
\psline[linewidth=0.05,linestyle=dotted](8.65,6.65)(9.5,6.65)
\psline[linewidth=0.05,linestyle=dotted](9.5,7.5)(9.5,6.65)
\psline[linewidth=0.05,linestyle=dotted](9.5,7.5)(8.65,7.5)
\psline[linewidth=0.05,linestyle=dotted](8.65,7.5)(8.65,6.65)
\psdot[dotsize=3pt 6](5,5)
\psdot[dotsize=3pt 6](6,6)
\psdot[dotsize=3pt 6](4,6)
\psdot[dotsize=3pt 6](4,4)
\psdot[dotsize=3pt 6](6,4)
\psdot[dotsize=3pt 6](2,4)
\psdot[dotsize=3pt 6](1,7)
\psdot[dotsize=3pt 6](3,8)
\psdot[dotsize=3pt 6](5,7)
\psdot[dotsize=3pt 6](6,8)
\psdot[dotsize=3pt 6](5,9)
\psdot[dotsize=3pt 6](3,9)
\psdot[dotsize=3pt 6](1,9)
\psdot[dotsize=3pt 6](8,8)
\psdot[dotsize=3pt 6](9,7)
\psdot[dotsize=3pt 6](9,5)
\psdot[dotsize=3pt 6](8,3)
\psdot[dotsize=3pt 6](7,2)
\psdot[dotsize=3pt 6](5,3)
\psdot[dotsize=3pt 6](5,1)
\psdot[dotsize=3pt 6](2,2)
\rput(0,3.5){\begin{Huge}$+$\end{Huge}}
\rput(0,8){\begin{Huge}$-$\end{Huge}}
\rput(10,7){\begin{Huge}$-$\end{Huge}}
\end{pspicture}
\caption{Full lines are the edges visited by discrete loops. 
Double lines are additionaly opened edges.
Dashed lines are edges left closed. 
Dotted contours surround clusters in $\mathfrak{C}'$.}
\label{Fig one}
\end{center}
\end{figure}

\section{Alternative proof of the coupling}
\label{SecAltProof}

In this section we prove directly, without using metric graphs, that the procedure described in Theorem \ref{ThmAltDescCoupl} bis provides a coupling between $\mathcal{L}_{1/2}$ and the Gaussian free field. We will denote by $\phi$ the field constructed by this procedure and $\psi$ a generic Gaussian free field on $\mathcal{G}$, so as to avoid confusion. 

Let $e_{1}=\lbrace x_{1},y_{1}\rbrace,\dots,
e_{n}=\lbrace x_{n},y_{n}\rbrace$ be $n$ different edges of $\mathcal{G}$. Let $\mathcal{G}^{(e_{1},\dots,e_{n})}$ be the graph obtained by removing the edges $e_{1},\dots,e_{n}$. $\mathcal{G}^{(e_{1},\dots,e_{n})}$ may not be connected. Let 
$\kappa^{(e_{1},\dots,e_{n})}$ be the killing measure on $V$ defined as
\begin{displaymath}
\kappa^{(e_{1},\dots,e_{n})}(x):=\kappa(x)+
\sum_{i=1}^{n}C(e_{i})(1_{x=x_{i}}+1_{x=y_{i}}).
\end{displaymath}
Let $(G^{(e_{1},\dots,e_{n})}(x,y))_{x,y\in V}$ be the Green's function of the Markov jump process on $\mathcal{G}^{(e_{1},\dots,e_{n})}$ with jump rates equal to conductances and killing rates given by $\kappa^{(e_{1},\dots,e_{n})}$. Let 
$(\psi^{(e_{1},\dots,e_{n})}_{x})_{x\in V}$ be the corresponding Gaussian free field on $\mathcal{G}^{(e_{1},\dots,e_{n})}$. Let $H$ be the energy functional
\begin{displaymath}
H(f):=\dfrac{1}{2}\bigg(\sum_{x\in V}\kappa(x)f_{x}^{2}+
\sum_{x,y\in V, \lbrace x,y\rbrace\in E}C(x,y)(f_{x}-f_{y})^{2}\bigg)
\end{displaymath}
and let
\begin{displaymath}
H^{(e_{1},\dots,e_{n})}(f):=H(f)+
\sum_{i=1}^{n}C(e_{i})f_{x_{i}}f_{y_{i}}.
\end{displaymath}
If $V$ is finite the distribution of $\psi$ is
\begin{displaymath}
\dfrac{1}{(2\pi)^{\frac{\vert V\vert}{2}}\det(G)^{\frac{1}{2}}}
e^{-H(f)}\prod_{x\in V}df_{x}
\end{displaymath}
and the distribution of $\psi^{(e_{1},\dots,e_{n})}$ is
\begin{displaymath}
\dfrac{1}{(2\pi)^{\frac{\vert V\vert}{2}}
\det(G^{(e_{1},\dots,e_{n})})^{\frac{1}{2}}}
e^{-H^{(e_{1},\dots,e_{n})}(f)}\prod_{x\in V}df_{x}.
\end{displaymath}
Conditional on $e_{i}\not\in\bigcup_{\mathcal{C}
\in\mathfrak{C}_{1/2}}\mathcal{C}$
for every $i\in\lbrace 1,\dots,n\rbrace$, $(\widehat{\mathcal{L}}_{1/2}^{x})_{x\in V}$ has the same law as
$\frac{1}{2}\psi^{(e_{1},\dots,e_{n}) 2}$. If $V$ is finite then
\begin{equation}
\label{SecAltProof: EqQuotDet}
\mathbb{P}\Big(\forall i\in\lbrace 1,\dots,n\rbrace,
e_{i}\not\in\bigcup_{\mathcal{C}\in\mathfrak{C}
_{1/2}}\mathcal{C}\Big)=
\dfrac{\det(G^{(e_{1},\dots,e_{n})})^{\frac{1}{2}}}
{\det(G)^{\frac{1}{2}}};
\end{equation}
see \cite{LeJanLemaire2012LoopClusters}.

\begin{lemma}
\label{LemGFFAbsEdges}
Assume that $V$ is finite. Let $e_{1}=\lbrace x_{1},y_{1}\rbrace,\dots,
e_{n}=\lbrace x_{n},y_{n}\rbrace$ be $n$ different edges of $\mathcal{G}$. For any bounded functional $F$ on the fields 
\begin{equation}
\label{SecAltProof: EqE1}
\mathbb{E}\left[F(\widehat{\mathcal{L}}_{1/2});
\forall i\in\lbrace 1,\dots,n\rbrace,
e_{i}\not\in\bigcup_{\mathcal{C}'\in\mathfrak{C}'}\mathcal{C}'\right]=
\mathbb{E}\left[\prod_{i=1}^{n}e^{-C(e_{i})(\vert \psi_{x_{i}}\psi_{y_{i}}\vert+\psi_{x_{i}}\psi_{y_{i}})}F\Big(
\dfrac{1}{2}\psi^{2}\Big)\right],
\end{equation}
\begin{multline}
\label{SecAltProof: EqE2}
\mathbb{E}\left[F(\widehat{\mathcal{L}}_{1/2});
E\setminus\bigcup_{\mathcal{C}'\in\mathfrak{C}'}\mathcal{C}'=
\lbrace e_{1},\dots,e_{n}\rbrace\right]=\\
\mathbb{E}\Bigg[\prod_{i=1}^{n}e^{-C(e_{i})(\vert \psi_{x_{i}}\psi_{y_{i}}\vert+\psi_{x_{i}}\psi_{y_{i}})}
\prod_{
\begin{scriptsize}
\begin{array}{c}
\lbrace x,y\rbrace\in \\ 
E\setminus\lbrace e_{1},\dots,e_{n}\rbrace
\end{array} 
\end{scriptsize}
}
1_{\psi_{x}\psi_{y}>0}
(1-e^{-2C(x,y)\vert \psi_{x}\psi_{y}\vert})
F\Big(\dfrac{1}{2}\psi^{2}\Big)\Bigg].
\end{multline}
\end{lemma}

\begin{proof}
We begin with the proof of \eqref{SecAltProof: EqE1}. Conditional on $e_{i}\not\in\bigcup_{\mathcal{C}
\in\mathfrak{C}_{1/2}}\mathcal{C}$
for every $i\in\lbrace 1,\dots,n\rbrace$, $(\widehat{\mathcal{L}}_{1/2}^{x})_{x\in V}$ has the same law as
$\frac{1}{2}\psi^{(e_{1},\dots,e_{n}) 2}$, that is to say
\begin{displaymath}
\mathbb{E}\Bigg[F(\widehat{\mathcal{L}}_{1/2})\Big\vert
\forall i\in\lbrace 1,\dots,n\rbrace,
e_{i}\not\in\bigcup_{\mathcal{C}
\in\mathfrak{C}_{1/2}}\mathcal{C}\Bigg]=
\mathbb{E}\left[F\Big(\dfrac{1}{2}\psi^{(e_{1},\dots,e_{n}) 2}
\Big)\right].
\end{displaymath}
Applying \eqref{SecAltProof: EqQuotDet} we get that
\begin{displaymath}
\mathbb{E}\Bigg[F(\widehat{\mathcal{L}}_{1/2});
\forall i\in\lbrace 1,\dots,n\rbrace,
e_{i}\not\in\bigcup_{\mathcal{C}
\in\mathfrak{C}_{1/2}}\mathcal{C}\Bigg]=
\dfrac{\det(G^{(e_{1},\dots,e_{n})})^{\frac{1}{2}}}
{\det(G)^{\frac{1}{2}}}
\mathbb{E}\left[F\Big(\dfrac{1}{2}\psi^{(e_{1},\dots,e_{n}) 2}
\Big)\right].
\end{displaymath}
But
\begin{multline*}
\dfrac{\det(G^{(e_{1},\dots,e_{n})})^{\frac{1}{2}}}
{\det(G)^{\frac{1}{2}}}
\mathbb{E}\left[F\Big(\dfrac{1}{2}\psi^{(e_{1},\dots,e_{n}) 2}
\Big)\right]\\=\dfrac{\det(G^{(e_{1},\dots,e_{n})})^{\frac{1}{2}}}
{\det(G)^{\frac{1}{2}}}\dfrac{1}{(2\pi)^{\frac{\vert V\vert}{2}}
\det(G^{(e_{1},\dots,e_{n})})^{\frac{1}{2}}}
\int e^{-H^{(e_{1},\dots,e_{n})}(f)}
F\Big(\dfrac{1}{2}f^{2}\Big)\prod_{x\in V}df_{x}\\=
\dfrac{1}{(2\pi)^{\frac{\vert V\vert}{2}}\det(G)^{\frac{1}{2}}}
\int e^{-H(f)}\prod_{i=1}^{n}e^{-C(e_{i})f_{x_{i}}f_{y_{i}}}
F\Big(\dfrac{1}{2}f^{2}\Big)\prod_{x\in V}df_{x}.
\end{multline*}
It follows that
\begin{displaymath}
\mathbb{E}\Bigg[F(\widehat{\mathcal{L}}_{1/2});
\forall i\in\lbrace 1,\dots,n\rbrace, e_{i}\not\in
\bigcup_{\mathcal{C}\in\mathfrak{C}_{1/2}}\mathcal{C}\Bigg]=
\mathbb{E}\left[\prod_{i=1}^{n}e^{-C(e_{i})\psi_{x_{i}}\psi_{y_{i}}}
F\Big(\dfrac{1}{2}\psi^{2}\Big)\right].
\end{displaymath}
Then (see \eqref{EqNiceExpr})
\begin{equation*}
\begin{split}
\mathbb{E}\Bigg[F(\widehat{\mathcal{L}}_{1/2})&;
\forall i\in\lbrace 1,\dots,n\rbrace,
e_{i}\not\in\bigcup_{\mathcal{C}'\in\mathfrak{C}'}\mathcal{C}'\Bigg]\\
&=\mathbb{E}\Bigg[\prod_{i=1}^{n}\exp\Big(-2C(e_{i})
\sqrt{\widehat{\mathcal{L}}_{1/2}^{x_{i}}
\widehat{\mathcal{L}}_{1/2}^{y_{i}}}\Big)
F(\widehat{\mathcal{L}}_{1/2});
\forall i\in\lbrace 1,\dots,n\rbrace, e_{i}\not\in
\bigcup_{\mathcal{C}\in\mathfrak{C}_{1/2}}\mathcal{C}\Bigg]
\\&=\mathbb{E}\left[\prod_{i=1}^{n}e^{-C(e_{i})
(\vert \psi_{x_{i}}\psi_{y_{i}}\vert+\psi_{x_{i}}\psi_{y_{i}})}
F\Big(\dfrac{1}{2}\psi^{2}\Big)\right].
\end{split}
\end{equation*}
For the proof of \eqref{SecAltProof: EqE2} we will use the inclusion-exclusion principle.
\begin{equation*}
\begin{split}
\mathbb{E}\Bigg[F(\widehat{\mathcal{L}}_{1/2});
E&\setminus\bigcup_{\mathcal{C}'\in\mathfrak{C}'}\mathcal{C}'=
\lbrace e_{1},\dots,e_{n}\rbrace\Bigg]\\=&
\sum_{
\begin{scriptsize}
\begin{array}{c}
A\subseteq E \\ 
\lbrace e_{1},\dots,e_{n}\rbrace\subseteq A
\end{array} 
\end{scriptsize}
}
(-1)^{\vert A\vert-n}
\mathbb{E}\Bigg[F(\widehat{\mathcal{L}}_{1/2});
\forall e\in A,
e\not\in\bigcup_{\mathcal{C}'\in\mathfrak{C}'}\mathcal{C}'\Bigg]
\\=&\sum_{
\begin{scriptsize}
\begin{array}{c}
A\subseteq E \\ 
\lbrace e_{1},\dots,e_{n}\rbrace\subseteq A
\end{array} 
\end{scriptsize}
}
(-1)^{\vert A\vert-n}
\mathbb{E}\Bigg[\prod_{\lbrace x,y\rbrace\in A}e^{-C(x,y)
(\vert \psi_{x}\psi_{y}\vert+\psi_{x}\psi_{y})}
F\Big(\dfrac{1}{2}\psi^{2}\Big)\Bigg]
\\=&\mathbb{E}\Bigg[\prod_{i=1}^{n}e^{-C(e_{i})(\vert \psi_{x_{i}}\psi_{y_{i}}\vert+\psi_{x_{i}}\psi_{y_{i}})}\\\times &
\prod_{\lbrace x,y\rbrace\in E\setminus\lbrace e_{1},\dots,e_{n}\rbrace}
\Big(1-e^{-C(x,y)(\vert \psi_{x}\psi_{y}\vert+\psi_{x}\psi_{y})}\Big)
F\Big(\dfrac{1}{2}\psi^{2}\Big)\Bigg].
\end{split}
\end{equation*}
But
\begin{displaymath}
1-e^{-C(x,y)(\vert \psi_{x}\psi_{y}\vert+\psi_{x}\psi_{y})}=
1_{\psi_{x}\psi_{y}>0}(1-e^{-2C(x,y)\vert \psi_{x}\psi_{y}\vert}).
\end{displaymath}
Thus we get \eqref{SecAltProof: EqE2}.
\end{proof}

\begin{proposition}
\label{PropAltProof}
The field $(\phi_{x})_{x\in V}$ constructed in Theorem \ref{ThmAltDescCoupl} bis has the law of a Gaussian free field on $\mathcal{G}$.
\end{proposition}

\begin{proof}
First we consider the case of $V$ being finite and use the identity 
\eqref{SecAltProof: EqE2}. Let $F$ be a bounded functional on fields. Given a subset of edges $A\subseteq E$, we will denote
by $\mathfrak{C}(A)$ the partition of $V$ obtained by removing from $\mathcal{G}$ the edges in $A$ and taking the connected components.
Let $\mathcal{S}_{A}(F)$ be the functional on non-negative fields defined as
\begin{displaymath}
\mathcal{S}_{A}(F)(f):=\dfrac{1}{2^{\vert \mathfrak{C}(A)\vert}}
\sum_{\sigma\in\lbrace -1,+1\rbrace^{\mathfrak{C}(A)}}
F(\sigma\sqrt{2f}),
\end{displaymath}
where $F(\sigma\sqrt{2f})$ means that we have made a choice of a sign which is the same on each equivalence class of the partition  $\mathfrak{C}(A)$. 

Let $e_{1}=\lbrace x_{1},y_{1}\rbrace,\dots,
e_{n}=\lbrace x_{n},y_{n}\rbrace$ be $n$ different edges of $\mathcal{G}$. By construction
\begin{multline}
\mathbb{E}\Bigg[F(\phi);E\setminus\bigcup_{\mathcal{C}'\in\mathfrak{C}'}
\mathcal{C}'=\lbrace e_{1},\dots,e_{n}\rbrace\Bigg]\\=
\mathbb{E}\Bigg[\mathcal{S}_{\lbrace e_{1},\dots,e_{n}\rbrace}(F)
(\widehat{\mathcal{L}}_{1/2});
E\setminus\bigcup_{\mathcal{C}'\in\mathfrak{C}'}\mathcal{C}'=
\lbrace e_{1},\dots,e_{n}\rbrace\Bigg].
\end{multline}
From \eqref{SecAltProof: EqE2} follows that this in turn equals
\begin{multline}
\label{SecAltProof: EqBigExp}
\mathbb{E}\Bigg[\prod_{i=1}^{n}e^{-C(e_{i})(\vert \psi_{x_{i}}\psi_{y_{i}}\vert+\psi_{x_{i}}\psi_{y_{i}})}
\prod_{
\begin{scriptsize}
\begin{array}{c}
\lbrace x,y\rbrace\in \\ 
E\setminus\lbrace e_{1},\dots,e_{n}\rbrace
\end{array} 
\end{scriptsize}
}1_{\psi_{x}\psi_{y}>0}
(1-e^{-2C(x,y)\vert \psi_{x}\psi_{y}\vert})
\\\times\mathcal{S}_{\lbrace e_{1},\dots,e_{n}\rbrace}(F)
\Big(\dfrac{1}{2}\psi^{2}\Big)\Bigg].
\end{multline}

In \eqref{SecAltProof: EqBigExp} the factor
\begin{displaymath}
\prod_{i=1}^{n}e^{-C(e_{i})\vert \psi_{x_{i}}\psi_{y_{i}}\vert}
\prod_{
\begin{scriptsize}
\begin{array}{c}
\lbrace x,y\rbrace\in \\ 
E\setminus\lbrace e_{1},\dots,e_{n}\rbrace
\end{array} 
\end{scriptsize}
}
(1-e^{-2C(x,y)\vert \psi_{x}\psi_{y}\vert})
\mathcal{S}_{\lbrace e_{1},\dots,e_{n}\rbrace}(F)
\Big(\dfrac{1}{2}\psi^{2}\Big)
\end{displaymath}
depends only on the absolute value $\vert\psi\vert$. Two other factors take in account the sign of $\psi$:
\begin{equation}
\label{SecAltProof: EqSign1}
\prod_{i=1}^{n}e^{-C(e_{i})\psi_{x_{i}}\psi_{y_{i}}}
\end{equation}
and
\begin{equation}
\label{SecAltProof: EqSign2}
\prod_{
\begin{scriptsize}
\begin{array}{c}
\lbrace x,y\rbrace\in \\ 
E\setminus\lbrace e_{1},\dots,e_{n}\rbrace
\end{array} 
\end{scriptsize}
}1_{\psi_{x}\psi_{y}>0}.
\end{equation}
The factor \eqref{SecAltProof: EqSign1} multiplied by the non normalized density $e^{-H(f)}$ of $\psi$ gives the non-normalized density
$e^{-H^{(e_{1},\dots,e_{n})}(f)}$ of $\psi^{{(e_{1},\dots,e_{n})}}$.
The Gaussian free field on $\mathcal{G}^{(e_{1},\dots,e_{n})}$.
$\psi^{{(e_{1},\dots,e_{n})}}$ is independent on each connected component of $\mathcal{G}^{(e_{1},\dots,e_{n})}$. The factor
\eqref{SecAltProof: EqSign2} means that we restrict to the event on which the field has constant sign on each connected component of
$\mathcal{G}^{(e_{1},\dots,e_{n})}$. But conditional on
$\psi^{{(e_{1},\dots,e_{n})}}$ having constant sign on each connected component of $\mathcal{G}^{(e_{1},\dots,e_{n})}$, these signs are independent on each connected component and $-$ and $+$ have equal probability $\frac{1}{2}$. This implies that \eqref{SecAltProof: EqBigExp} equals 
\begin{displaymath}
\mathbb{E}\Bigg[\prod_{i=1}^{n}e^{-C(e_{i})(\vert \psi_{x_{i}}\psi_{y_{i}}\vert+\psi_{x_{i}}\psi_{y_{i}})}
\prod_{
\begin{scriptsize}
\begin{array}{c}
\lbrace x,y\rbrace\in \\ 
E\setminus\lbrace e_{1},\dots,e_{n}\rbrace
\end{array} 
\end{scriptsize}
}1_{\psi_{x}\psi_{y}>0}
(1-e^{-2C(x,y)\vert \psi_{x}\psi_{y}\vert})F(\psi)\Bigg].
\end{displaymath}
Then summing on all possible values of $E\setminus\bigcup_{\mathcal{C}'\in\mathfrak{C}'}
\mathcal{C}'$ we get $\mathbb{E}[F(\phi)]=\mathbb{E}[F(\psi)]$ and deduce that $\phi$ and $\psi$ are
equidistributed.

For the case of infinite $V$ we approximate the graph $\mathcal{G}$ by an increasing sequence of finite connected sub-graphs.
Let $x_{0}\in V$. Let $V_{n}$ be the set of vertices separated from $x_{0}$ by at most $n$ edges. For $n\geq 1$ let $E_{n}$ be the set of edges either connecting two vertices in $V_{n-1}$ or a vertex in 
$V_{n}\setminus V_{n-1}$ to a vertex in $V_{n-1}$. 
$\mathcal{G}_{n}:=(V_{n},E_{n})$ is a connected sub-graph of $\mathcal{G}$. We consider the Markov jump process on $\mathcal{G}_{n}$ with transition rates given by the conductances restricted to $E_{n}$,
the killing measure $\kappa$ restricted to $V_{n-1}$ and an additional instant killing at reaching $V_{n}\setminus V_{n-1}$. Let $(G^{V_{n-1}}(x,y))_{x,y\in V_{n-1}}$ be the corresponding Green's function and 
$(\psi^{V_{n-1}}_{x})_{x\in V_{n-1}}$ the corresponding Gaussian free field. The associated Poisson ensemble of loops of parameter $\frac{1}{2}$ is $\lbrace\gamma\in\mathcal{L}
_{1/2}\vert\gamma~\text{stays in}~V_{n-1}\rbrace$. Let
$(\phi^{V_{n-1}}_{x})_{x\in V_{n-1}}$ be the field obtained by applying
the procedure described in Theorem \ref{ThmAltDescCoupl} bis to
$\lbrace\gamma\in\mathcal{L}
_{1/2}\vert\gamma~\text{stays in}~V_{n-1}\rbrace$. 
As shown previously $\phi^{V_{n-1}}$ has same law as $\psi^{V_{n-1}}$. Moreover $\psi^{V_{n-1}}$ converges in law to
$\psi$. To conclude that $\phi$ and $\psi$ have same law,
we need to show the convergence in law of
$\phi^{V_{n-1}}$  to $\phi$
(i.e. the convergence in law of
finite-dimensional marginals), which we detail below.

We will couple
the $(\phi^{V_{n-1}})_{n\geq 1}$ and 
$\phi$ on the same probability space.
As already described, we use one Poisson ensemble of loops
$\mathcal{L}_{1/2}$, and its subsets, 
for $\phi$ and for $\phi^{V_{n-1}}$.
In this way,
for every $x\in V$,
a.s. the subset of loops
\begin{displaymath}
\lbrace\gamma\in
\mathcal{L}_{1/2}\vert
\gamma \text{ visits } x
\text{ and }
V\setminus V_{n}\rbrace
\end{displaymath}
is empty for $n$ large enough, and the sequence
$(\vert\phi^{V_{n-1}}_{x}\vert)_{n\geq 1}$ is 
thus stationary for $n$ large enough, equal to $\vert \phi_{x}\vert$.
Further, we consider $(U_{e})_{e\in E}$ a family of i.i.d. uniform r.v.'s in $(0,1)$, independent from $\mathcal{L}_{1/2}$.
The Bernoulli r.v.'s in 
Theorem \ref{ThmAltDescCoupl} bis are obtained as
\begin{displaymath}
1_{U_{\{x,y\}} < 1-\exp(-C(x,y)\vert \phi_{x}\phi_{y}\vert)},
\text{ respectively }
1_{U_{\{x,y\}} < 1-\exp(-C(x,y)
\vert \phi^{V_{n-1}}_{x}\phi^{V_{n-1}}_{y}\vert)}.
\end{displaymath}
For fixed $\{ x,y\}\in E$, a.s. the Bernoulli r.v.'s are non-decreasing in $n$ and stationary for $n$ large enough, because
$\vert \phi^{V_{n-1}}_{x}\vert$ and
$\vert \phi^{V_{n-1}}_{y}\vert$ are.
$\mathcal{L}_{1/2}$ and $(U_{e})_{e\in E}$
determine the clusters
$\mathfrak{C}'$ in $V$ and 
$\mathfrak{C}'_{n-1}$ in $V_{n-1}$.
By construction there is a monotonicity:
each cluster in
$\mathfrak{C}'_{n-1}$ is contained in a cluster in
$\mathfrak{C}'_{n}$. 
Indeed, by going from $n-1$ to $n$, one adds more loops and does not decrease the Bernoulli r.v.'s above.
Moreover, if two vertices $x,y\in V$ are in the same cluster in
$\mathfrak{C}'$, then they are connected by a finite number of loops in $\mathcal{L}_{1/2}$ and a finite number of additionally opened edges.
Thus, $x$ and $y$ are in the same cluster in $\mathfrak{C}'_{n-1}$
for $n$ large enough.
Finally, we choose $(x_{i})_{i\geq 1}$ an enumeration of
$V\setminus\{x_{0}\}$, such that the graph distance between
$x_{i}$ and $x_{0}$ is non-decreasing, and 
take $(\check{\sigma}_{i})_{i\in\mathbb{N}}$ an i.i.d. family of 
uniform r.v. in $\lbrace -1,+1\rbrace$, independent from
$(\mathcal{L}_{1/2},(U_{e})_{e\in E})$.
For $\mathcal{C}'$ a cluster in $\mathfrak{C}'$,
respectively in $\mathfrak{C}'_{n-1}$, we set
\begin{displaymath}
\sigma(\mathcal{C}')=
\check{\sigma}_{\min\{i\in\mathbb{N}\vert x_{i}\in \mathcal{C}'\}},
\text{ respectively }
\sigma_{n-1}(\mathcal{C}')=
\check{\sigma}_{\min\{i\in\mathbb{N}\vert x_{i}\in \mathcal{C}'\}}.
\end{displaymath}
The signs $\sigma$ on $\mathfrak{C}'$,
respectively
$\sigma_{n-1}$ on $\mathfrak{C}_{n-1}'$,
complete the construction of
$\phi$, respectively of $\phi^{V_{n-1}}$. 
In this construction, for every $x\in V$, a.s.
$\operatorname{sign}(\phi^{V_{n-1}}_{x})
=\operatorname{sign}(\phi_{x})$ for
$n$ large enough, and thus
$\phi^{V_{n-1}}_{x}=\phi_{x}$
for $n$ large enough.
\end{proof}

\section{Application to percolation by loops}
\label{SecPerco}

In this section we consider the lattices
\begin{itemize}
\item $\mathbb{Z}^{2}$ with uniform conductances and a non-zero uniform killing measure,
\item the discrete half-plane $\mathbb{Z}\times\mathbb{N}$ with instantaneous killing on the boundary $\mathbb{Z}\times\lbrace 0\rbrace$
and no killing elsewhere,
\item $\mathbb{Z}^{d}$, $d\geq 3$, with uniform conductances and no killing measure,
\end{itemize}
and show that there is no infinite loop cluster in $\mathcal{L}_{1/2}$. Obviously there cannot be such an infinite cluster if the Gaussian free field only has bounded sign clusters, which is the case for $\mathbb{Z}^{2}$ with uniform conductances and a non-zero uniform killing measure (see Theorem $14.3$ in \cite{HaggstromJonasson2006UniqPerc}). However on $\mathbb{Z}^{d}$ for $d$ sufficiently large the Gaussian free field has infinite sign clusters, one of each sign, at is it believed that is the case for all $d\geq 3$ (\cite{RodriguezSznitman2013PercGFF}). But at the level of the metric graph there are no unbounded sign clusters of the free field.

The uniqueness of an infinite cluster of loops on $\mathbb{Z}^{d}$, $d\geq 3$ and on $\mathbb{Z}^{2}$ with uniform killing measure was shown applying Burton-Keane's argument in \cite{ChangSapozhnikov2014PercLoops}. Next we adapt this argument to the case of loops on the discrete half-plane.

\begin{proposition}
\label{PropUniqHalfPlane}
On the discrete half-plane $\mathbb{Z}\times\mathbb{N}$ with instantaneous killing on the boundary $\mathbb{Z}\times\lbrace 0\rbrace$, a.s. $\mathcal{L}_{1/2}$ has at most one infinite cluster.
\end{proposition}

\begin{proof}
The general layout of the proof is the same as for the i.i.d. Bernoulli percolation. See Section $8.2$ in \cite{Grimmett1999Percolation}.
The law of $\mathcal{L}_{1/2}$ is ergodic for the horizontal translations and hence the number of infinite clusters in $\mathfrak{C}_{1/2}$ is a.s. constant. The next step is to show that this constant can only be $0$, $1$ or $+\infty$. This can be proved similarly to the i.i.d. Bernoulli percolation case and we omit it. Then one has to rule out the case of infinitely many infinite clusters.

For $a\in\mathbb{N}$ let
\begin{displaymath}
\mathcal{L}_{1/2}^{>a}:=
\lbrace \gamma\in \mathcal{L}_{1/2}\vert 
\operatorname{Range}(\gamma)
\subseteq \mathbb{Z}\times [a+1,+\infty)\rbrace.
\end{displaymath}
$\mathcal{L}_{1/2}^{>0}=\mathcal{L}_{1/2}$ and all the
$\mathcal{L}_{1/2}^{>a}$ have the same law up to a vertical translation. A vertex $(x_{1},a+1)\in\mathbb{Z}\times\mathbb{N}^{\ast}$
will be an \textit{upper trifurcation} if it is contained in an infinite cluster of $\mathcal{L}_{1/2}^{>a}$ and if this vertex and adjacent edges are removed the cluster splits in at least three infinite clusters. Every vertex of $\mathbb{Z}\times\mathbb{N}^{\ast}$ has equal probability to be an upper trifurcation. Let it be $p_{3}$. If with positive probability $\mathcal{L}_{1/2}$ has at least three infinite clusters then a vertex in $\mathbb{Z}\times\lbrace 1\rbrace$ has a positive probability to be an upper trifurcation. This can be proved in the similar way as in i.i.d. Bernoulli case. Consequently $p_{3}>0$.

Let $\mathcal{T}_{n}$ be the set of upper trifurcations in 
$[-n,n]\times [1,n]$. Let $(z_{i})_{1\leq i\leq N_{n}}$ be an enumeration of $\mathcal{T}_{n}$ such that the sequence of second coordinates of $z_{i}$, $(a_{i}+1)_{1\leq i\leq N_{n}}$, is non-increasing. Given $z_{i}$, there are three simple paths $c_{1}(z_{i})$, $c_{2}(z_{i})$ and $c_{3}(z_{i})$ that connect $z_{i}$ to three different vertices on
\begin{multline*}
\partial ([-n-1,n+1]\times [1,n+1])=\\
\lbrace -n-1\rbrace\times [1,n+1]\cup\lbrace n+1\rbrace\times [1,n+1]
\cup [-n-1,n+1]\times \lbrace n+1\rbrace
\end{multline*}
that do not intersect outside $z_{i}$ and such that 
$c_{1}(z_{i})\setminus \lbrace z_{i}\rbrace$, $c_{2}(z_{i})\setminus \lbrace z_{i}\rbrace$ and $c_{3}(z_{i})\setminus \lbrace z_{i}\rbrace$ are contained in three different clusters induced by the clusters of $\mathcal{L}_{1/2}^{>a_{i}}$ after deleting the vertex $z_{i}$. For $i\geq 2$, two different paths $c_{j}(z_{i})\setminus \lbrace z_{i}\rbrace$ and 
$c_{j'}(z_{i})\setminus \lbrace z_{i}\rbrace$ cannot intersect the same connected component of
\begin{displaymath}
\bigcup_{1\leq i'\leq i-1}(c_{1}(z_{i'})\cup c_{2}(z_{i'})
\cup c_{3}(z_{i'}))
\end{displaymath}
because the set above is covered by the loops in $\mathcal{L}_{1/2}^{>a_{i}}$. Then as in Burton-Keane's proof one sets $\tilde{c}_{j}(z_{1})=c_{j}(z_{1})$ and iteratively constructs the family of simple paths $(\tilde{c}_{j}(z_{i}))_{1\leq j\leq 3, 2\leq i\leq N_{n}}$ where the path $\tilde{c}_{j}(z_{i})$ starts from $z_{i}$ as $c_{j}(z_{i})$ and as soon as it meets a path $\tilde{c}$ from the family
$(\tilde{c}_{j'}(z_{i'}))_{1\leq j'\leq 3, 1\leq i'\leq i-1}$ it continues as $\tilde{c}$. The graph formed by the paths
$(\tilde{c}_{j}(z_{i}))_{1\leq j\leq 3, 1\leq i\leq N_{n}}$ has no cycles, its leaves (vertices of degree $1$) are contained in $\partial ([-n-1,n+1]\times [1,n+1])$ and the vertices $z_{i}$ have degree $3$ at least. Thus
\begin{displaymath}
\sharp \mathcal{T}_{n} \leq \sharp \partial ([-n-1,n+1]\times [1,n+1])=4n+3.
\end{displaymath}
The expectation of $\sharp \mathcal{T}_{n}$ cannot grow as fast as $n^{2}$ hence $p_{3}=0$.
\end{proof}

Next we give a simple upper bound for the probability of two vertices belonging to the same cluster of $\mathcal{L}_{1/2}$. This is an inequality that holds on all graphs and not specifically on periodic ones as considered previously in this section.

\begin{proposition}
\label{PropMajClust}
Let $x,y\in V$. Let
\begin{displaymath}
g(x,y):=\dfrac{G(x,y)}{\sqrt{G(x,x)G(y,y)}}.
\end{displaymath}
\begin{multline}
\label{SecPerco: EqBound}
\mathbb{P}\left(x~and~y~belong~to~the~same~cluster~of~
\mathcal{L}_{1/2}\right)\leq\\
\mathbb{P}\left(x~and~y~belong~to~the~same~cluster~of~
\widetilde{\mathcal{L}}_{1/2}\right)=
\dfrac{2}{\pi}
\arcsin(g(x,y)).
\end{multline}
\end{proposition}

\begin{proof}
Consider the set of extended clusters $\mathfrak{C}'$. The probability that $x$ and $y$ belong to the same cluster in $\mathfrak{C}'$ is exactly
\begin{displaymath}
\mathbb{E}\left[\operatorname{sign}(\phi_{x})\operatorname{sign}(\phi_{y})\right].
\end{displaymath}
In our coupling if $x$ and $y$ belong to the same cluster in $\mathfrak{C}'$ then the product 
$\operatorname{sign}(\phi_{x})\operatorname{sign}(\phi_{y})$ equals $1$, and if this is not the case 
$\operatorname{sign}(\phi_{x})\operatorname{sign}(\phi_{y})$ equals either $1$ or $-1$ each with probability $\frac{1}{2}$.

One must to check that
\begin{displaymath}
\mathbb{E}\left[\operatorname{sign}(\phi_{x})\operatorname{sign}(\phi_{y})\right]=\dfrac{2}{\pi}
\arcsin(g(x,y)).
\end{displaymath}
Let $Z_{1}$ and $Z_{2}$ be two independent standard centred Gaussian r.v.'s. We have the equalities in law
\begin{displaymath}
(\phi_{x},\phi_{y})\stackrel{(\text{law})}{=}
(\sqrt{G(x,x)}Z_{1},\sqrt{G(y,y)}
(g(x,y)Z_{1}+\sqrt{1-g(x,y)^{2}}Z_{2})),
\end{displaymath}
\begin{displaymath}
(\operatorname{sign}(\phi_{x}),\operatorname{sign}(\phi_{y}))\stackrel{(\text{law})}{=}
(\operatorname{sign}(Z_{1}),\operatorname{sign}(g(x,y)Z_{1}+\sqrt{1-g(x,y)^{2}}Z_{2})).
\end{displaymath}
Then
\begin{displaymath}
\mathbb{E}\left[\operatorname{sign}(\phi_{x})\operatorname{sign}(\phi_{y})\right]=
\mathbb{P}\left(\dfrac{\vert Z_{2}\vert}{\vert Z_{1}\vert}\leq
\dfrac{g(x,y)}{\sqrt{1-g(x,y)^{2}}}\right).
\end{displaymath}
$\dfrac{Z_{2}}{Z_{1}}$ follows the Cauchy distribution
\begin{displaymath}
\dfrac{1}{\pi}\dfrac{dz}{1+z^{2}}.
\end{displaymath}
Thus
\begin{displaymath}
\mathbb{P}\left(\dfrac{\vert Z_{2}\vert}{\vert Z_{1}\vert}\leq
\dfrac{g(x,y)}{\sqrt{1-g(x,y)^{2}}}\right)=\dfrac{2}{\pi}
\arctan\left(\dfrac{g(x,y)}{\sqrt{1-g(x,y)^{2}}}\right)=\dfrac{2}{\pi}\arcsin(g(x,y)).
\qedhere
\end{displaymath}
\end{proof}

In the case of a graph $\mathbb{Z}^{d}$ ($d\geq 2$) with positive constant killing measure, inequality \eqref{SecPerco: EqBound} ensures an exponential decay of cluster size distribution. For $\mathbb{Z}^{d}$ ($d\geq 3$) with no killing inequality 
\eqref{SecPerco: EqBound} implies
\begin{displaymath}
\mathbb{P}\left(x~\text{and}~y~\text{belong to the same cluster of}~
\mathcal{L}_{1/2}\right)=
O\left(\dfrac{1}{\vert y-x\vert^{d-2}}\right).
\end{displaymath}
However this bound is certainly not sharp and one expects that for $d\geq 5$,
\begin{displaymath}
\mathbb{P}\left(x~\text{and}~y~\text{belong to the same cluster of}~
\mathcal{L}_{1/2}\right)=
O\left(\dfrac{1}{\vert y-x\vert^{2(d-2})}\right);
\end{displaymath}
see Proposition $5.3$ in \cite{ChangSapozhnikov2014PercLoops}. This also means that the percolation by discrete loops on periodic lattices and the percolation by continuous loops on the corresponding metric graphs behave differently.

\begin{proof}[Proof of Theorem \ref{ThmPerco}]
Assume that $\mathcal{L}_{1/2}$ has an infinite cluster. Let $\mathcal{C}_{\infty}$ be this infinite cluster. Let $x$ be a vertex and
\begin{displaymath}
\theta(x):=\mathbb{P}(x\in\mathcal{C}_{\infty}).
\end{displaymath}
Let $u_{1}$ be the unit vector corresponding to the first coordinate
\begin{displaymath}
u_{1}=(1,0,\dots,0). 
\end{displaymath}
Let $x_{n}:=x+nu_{1}$. From the invariance under translation by $u_{1}$ it follows that $\theta(x_{n})=\theta(x)$.
\begin{displaymath}
\mathbb{P}\left(x~\text{and}~x_{n}~\text{belong to the same cluster of}~
\mathcal{L}_{1/2}\right)\geq
\mathbb{P}(x\in\mathcal{C}_{\infty},x_{n}\in\mathcal{C}_{\infty}).
\end{displaymath}
$\mathcal{L}_{1/2}$ satisfies Harris-FKG inequality (\cite{LeJanLemaire2012LoopClusters}). Thus
\begin{displaymath}
\mathbb{P}(x\in\mathcal{C}_{\infty},x_{n}\in\mathcal{C}_{\infty})\geq
\theta(x)\theta(x_{n})=\theta(x)^{2}.
\end{displaymath}
It follows that
\begin{displaymath}
\theta(x)^{2}\leq\dfrac{2}{\pi}\arcsin(g(x,x_{n})).
\end{displaymath}
Letting $n$ go to $+\infty$ we get that $\theta(x)=0$.
\end{proof}

Let $d\geq 3$. Let $\widetilde{\mathbb{Z}}^{d}$ be the metric graph associated to the graph $\mathbb{Z}^{d}$. All edges have length $\frac{1}{2}$. We consider the Gaussian free field $(\phi_{z})_{z\in\widetilde{\mathbb{Z}}^{d}}$ on $\widetilde{\mathbb{Z}}^{d}$ and the following dependent percolation model on the edges of $\mathbb{Z}^{d}$: Let $\omega$ be the random configuration on the edges of $\mathbb{Z}^{d}$ with $\omega_{e}=1$ ($e$ is open) if $\vert\phi\vert$ has no zeros on $I_{e}$ and $\omega_{e}=0$ ($e$ is closed) otherwise. The set of clusters of $\omega$ is exactly $\mathfrak{C}'$ which appears in the coupling of Theorem \ref{ThmAltDescCoupl} bis. The free field on the metric graph has an unbounded sign cluster if an only if there is an infinite cluster in $\mathfrak{C}'$, as the sign clusters of $\phi$ that are contained inside the intervals $I_{e}$ corresponding to the edges are all bounded.
We will show that this cannot happen. We will follow the same pattern as for the proof of Theorem \ref{ThmPerco}: first show that $\mathfrak{C}'$ can contain at most one infinite cluster, the show that $\omega$ satisfies the Harris-FKG inequality and conclude using inequality \eqref{SecPerco: EqBound}.

\begin{lemma}
\label{LemUniqGFFCable}
Let $d\geq 3$. We consider the metric graph loops $\widetilde{\mathcal{L}}_{1/2}$ on the metric graph $\widetilde{\mathbb{Z}}^{d}$ associated to $\mathbb{Z}^{d}$. $\mathfrak{C}'$ is the trace on the vertices and edges of $\mathbb{Z}^{d}$ of the clusters of $\widetilde{\mathcal{L}}_{1/2}$. With probability one $\mathfrak{C}'$ has at most one infinite cluster.
\end{lemma}

\begin{proof}
According to Theorem $1$ in \cite{GandolfiKeaneNewman1992UniqInfClust}, the uniqueness of the infinite clusters is implied by translation invariance and positive finite energy property. We need only show the finite energy property,
\begin{equation}
\label{SecPerco: EqPosFiniteNRJ}
\mathbb{P}(\omega_{e}=1\vert (\omega_{f},~f~\text{is an edge of}~\mathbb{Z}^{d}~\text{and}~f\neq e))>0~\text{a.s.}
\end{equation}
Let $e=\lbrace x,y\rbrace$ be an edge. We see $(\frac{1}{2}\phi_{z}^{2})_{z\in\widetilde{\mathbb{Z}}^{d}}$ as the occupation field of continuous loops $\widetilde{\mathcal{L}}_{1/2}$. The loops inside $I_{e}$ and the excursions inside $I_{e}$ from $x$ to $x$ and $y$ to $y$ that do not cross entirely $I_{e}$ are independent of  $(\frac{1}{2}\phi_{z}^{2})_{z\in\widetilde{\mathbb{Z}}^{d}\setminus I_{e}}$ conditional on $\vert\phi_{x}\vert$ and $\vert\phi_{y}\vert$. Thus, according to the computations made in Section \ref{SecAltDesc},
\begin{displaymath}
\mathbb{P}\Big(\omega_{e}=1\Big\vert \Big(\frac{1}{2}\phi_{z}^{2}\Big)_{z\in\widetilde{\mathbb{Z}}^{d}\setminus I_{e}}\Big)\geq 1-e^{-\vert \phi_{x}\phi_{y}\vert}.
\end{displaymath}
Hence
\begin{multline}
\label{SecPerco: EqMajCondProb}
\mathbb{P}(\omega_{e}=1\vert (\omega_{f},~f~\text{is an edge of}~\mathbb{Z}^{d}~\text{and}~f\neq e))\geq\\
\mathbb{E}\left[1-e^{-\vert \phi_{x}\phi_{y}\vert}\vert (\omega_{f},~f~\text{is an edge of}~\mathbb{Z}^{d}~
\text{and}~f\neq e)\right].
\end{multline}
Since $\vert \phi_{x}\phi_{y}\vert>0$ a.s., the right-hand side in \eqref{SecPerco: EqMajCondProb} is a.s. non-zero and the condition \eqref{SecPerco: EqPosFiniteNRJ} is satisfied.
\end{proof}

\begin{lemma}
\label{LemCableFKG}
Le random configuration $\omega$ on edges of $\mathbb{Z}^{d}$ satisfies the Harris-FKG inequality: given $A_{1}(\omega)$ and $A_{2}(\omega)$ two increasing events
\begin{equation}
\label{SecPerco: EqFKG}
\mathbb{P}(A_{1}(\omega),A_{2}(\omega))\geq\mathbb{P}(A_{1}(\omega))
\mathbb{P}(A_{2}(\omega)) 
\end{equation}
\end{lemma}

\begin{proof}
We see the field $(\frac{1}{2}\phi_{y}^{2})_{y\in\widetilde{\mathbb{Z}}^{d}}$ as the occupation field of the metric graph loops $\widetilde{\mathcal{L}}_{1/2}$. If the events $A_{1}(\omega)$ and $A_{2}(\omega)$ are increasing in the sense that opening more edges in $\omega$ only helps their occurrence, then these events are also increasing in the sense that they are stable by adding more loops to $\widetilde{\mathcal{L}}_{1/2}$. The inequality \eqref{SecPerco: EqFKG} follows from the FKG inequality for Poisson point processes (Lemma $2.1$ in \cite{Janson1984FKGPPP}).
\end{proof}

\begin{proposition}
\label{PropFiniteClusters}
Let $d\geq 3$. We consider the metric graph loops $\widetilde{\mathcal{L}}_{1/2}$ on the metric graph $\widetilde{\mathbb{Z}}^{d}$ associated to $\mathbb{Z}^{d}$. $\mathfrak{C}'$ is the trace on the vertices and edges of $\mathbb{Z}^{d}$ of the clusters of $\widetilde{\mathcal{L}}_{1/2}$. With probability one $\mathfrak{C}'$ has only finite clusters.
\end{proposition}

\section{Random interlacements and level sets of the Gaussian free field}
\label{SecInterlacement}

Let $d\geq 3$. As in Section \ref{SecCouplComplex} we consider the metric graph $\widetilde{\mathbb{Z}}^{d}$ associated to the graph $\mathbb{Z}^{d}$. All edges have length $\frac{1}{2}$. We construct a continuous version $\widetilde{\mathcal{I}}^{u}$ of the random interlacement of level $u$ on the metric graph $\widetilde{\mathbb{Z}}^{d}$.
First we sample $\mathcal{I}^{u}$. Given a path $w$ in $\mathcal{I}^{u}$ we replace each jump from a vertex to its neighbour by a Brownian excursion inside the linking edge and we add Brownian excursions from a vertex visited by $w$ to itself inside adjacent edges such that the local time on the vertex equals the time $w$ spends in it (as in \eqref{SecCouplComplex: EqIntensity} for loops). By construction $\mathcal{I}^{u}$ is the restriction of $\widetilde{\mathcal{I}}^{u}$ to the vertices. $\widetilde{\mathcal{I}}^{u}$ has an occupation field 
$(L^{y}(\widetilde{\mathcal{I}}^{u}))_{y\in\widetilde{\mathbb{Z}}^{d}}$ which is continuous (because the occupation field of the Brownian excursions is) and its restriction to the vertices is $(L^{x}(\mathcal{I}^{u}))_{x\in\mathbb{Z}^{d}}$. We will show that the isomorphism \eqref{SecIntro: EqIsoSz} also holds in the continuous setting on $\widetilde{\mathbb{Z}}^{d}$. To this end we will use the approximation scheme of random interlacement by excursions that appeared in \cite{Sznitman2012Isomorphism}.

Let $K$ be a finite subset of $\mathbb{Z}^{d}$. Let $\mathcal{I}_{K}^{u}$ be the set of trajectories in $\mathcal{I}^{u}$ that visit $K$. Given such a trajectory $w$ we will denote by $(w^{K}(t))_{t\geq 0}$ the trajectory obtained by setting the origin of times at the entrance time of $w$ in $K$ and running $w$ onward from this time. Conditional on $w^{K}_{0}$, $(w^{K}(t))_{t\geq 0}$ is a Markov jump process on $\mathbb{Z}^{d}$.

Let $\mathcal{G}_{n}$ be the (discrete) graph obtained from the subgraph $[-n,n]^{d}$ of $\mathbb{Z}^{d}$ by identifying to one vertex $x_{\ast}$ the boundary of $[-n,n]^{d}$, that is the vertices
\begin{displaymath}
\lbrace (x_{1},x_{2},\dots,x_{d})\in[-n,n]^{d}\vert \exists ! i\in\lbrace 1,\dots,d\rbrace, 
\vert x_{i}\vert=n\rbrace.
\end{displaymath}
Between any two distinct adjacent vertices in $\mathcal{G}_{n}$ the conductance is $1$. Let $(X^{n}_{t})_{t\geq 0}$ be the recurrent Markov jump process on $\mathcal{G}_{n}$ starting from $x_{\ast}$. $X^{n}$ jumps away from $x_{\ast}$ with rate $2d(2n-1)^{d-1}$. Let
\begin{displaymath}
\tau^{n}_{u}:=\inf\bigg\lbrace t\geq 0\Big\vert\int_{0}^{t} 1_{X^{n}_{s}=x_{\ast}}ds=u\bigg\rbrace .
\end{displaymath}
There are two sequences $(D_{j}^{n})_{j\geq 1}$ and $(R_{j}^{n})_{j\geq 1}$ with
\begin{displaymath}
0<D_{1}^{n}<R_{1}^{n}<D_{2}^{n}<R_{2}^{n}<\dots<D_{j}^{n}<R_{j}^{n}<\cdots
\end{displaymath}
of successive departure and return times of $X^{n}$ from and to $x_{\ast}$. By convention we set $R_{0}^{n}=0$. $X^{n}$ is outside $x_{\ast}$ on time intervals $[D_{j}^{n},R_{j}^{n})$ and in $x_{\ast}$ on intervals $[R_{j-1}^{n},D_{j}^{n})$. Let
\begin{displaymath}
j_{u}^{n}:=\max\lbrace j\geq 0\vert R_{j}^{n}<\tau^{n}_{u}\rbrace.
\end{displaymath}
Let $K$ be a subset of $[-(n-1),n-1]^{d}$. Let
\begin{displaymath}
J_{u,K}^{n}:=\lbrace j\in\lbrace 1,\dots,j_{u}^{n}\rbrace\vert X^{n}~visits~K~on~[D_{j}^{n},R_{j}^{n})\rbrace.
\end{displaymath}
For $j\in J_{u,K}^{n}$ we define the stopping time $T_{K,j}^{n}$:
\begin{displaymath}
T_{K,j}^{n}:=\inf\lbrace t\in [D_{j}^{n},R_{j}^{n})\vert X^{n}_{t}\in K\rbrace.
\end{displaymath}
Conditional on $j\in J_{u,K}^{n}$ and on the value of $X^{n}_{T_{K,j}^{n}}$, the trajectory
\begin{displaymath}
(X^{n}_{T_{K,j}^{n}+t})_{0\leq t\leq R_{j}^{n}-T_{K,j}^{n}}
\end{displaymath} 
is a Markov jump process on $\mathcal{G}_{n}$ run until hitting $x_{\ast}$ or equivalently a Markov jump process on $\mathbb{Z}^{d}$ run until hitting the boundary of $[-n,n]^{d}$.

The next approximation result was shown in the first proof of Theorem $2.1$ in \cite{Sznitman2012Isomorphism}:

\begin{lemma}
\label{SecInterlacement: LemApproxScheme}
Let $K$ be a finite subset of $\mathbb{Z}^{d}$. The set of points
\begin{displaymath}
\lbrace X^{n}_{T_{K,j}^{n}}\vert j\in J_{u,K}^{n}\rbrace
\end{displaymath}
converges in law as $n\rightarrow +\infty$ to
\begin{displaymath}
\lbrace w^{K}(0)\vert w\in \mathcal{I}_{K}^{u}\rbrace.
\end{displaymath}
The set of trajectories
\begin{displaymath}
\left\lbrace(X^{n}_{T_{K,j}^{n}+t})_{0\leq t\leq R_{j}^{n}-T_{K,j}^{n}}\vert j\in J_{u,K}^{n}\right\rbrace
\end{displaymath}
converges in law to the set of trajectories
\begin{displaymath}
\lbrace (w^{K}(t))_{t\geq 0}\vert w\in \mathcal{I}_{K}^{u}\rbrace.
\end{displaymath}
\end{lemma}

Let $\widetilde{\mathcal{G}}_{n}$ be the metric graph associated to the graph $\mathcal{G}_{n}$. Let $(B^{\widetilde{\mathcal{G}}_{n}}_{t})_{t\geq 0}$ be a Brownian motion on $\widetilde{\mathcal{G}}_{n}$ starting from $x_{\ast}$ and $(L^{y}_{t}(B^{\widetilde{\mathcal{G}}_{n}}))_{t\geq 0, y\in \widetilde{\mathcal{G}}_{n}}$ its family of local times. Let
\begin{displaymath}
\tilde{\tau}_{u}^{n}:=\inf\lbrace t\geq 0\vert L^{x_{\ast}}_{t}(B^{\widetilde{\mathcal{G}}_{n}})>u\rbrace.
\end{displaymath}
For $r\in\mathbb{N}^{\ast}$ we denote by $\widetilde{\Lambda}_{r}$ the metric graph associated to the subgraph 
$[-r,r]^{d}$ of $\mathbb{Z}^{d}$ (without identification of boundary points).

\begin{lemma}
\label{SecInterlacement: LemApproxCable}
For all $r\in\mathbb{N}^{\ast}$ the occupation field $\Big(L^{y}_{\tilde{\tau}_{u}^{n}}(B^{\widetilde{\mathcal{G}}_{n}})\Big)_{y\in \widetilde{\Lambda}_{r}}$ converges in law as $n\rightarrow +\infty$ to 
$\Big(L^{y}(\widetilde{\mathcal{I}}^{u})\Big)_{y\in \widetilde{\Lambda}_{r}}$.
\end{lemma}

\begin{proof}
The Markov jump process $X^{n}$ on $\mathcal{G}_{n}$ is obtained from the Brownian motion $B^{\widetilde{\mathcal{G}}_{n}}$ through a time change by the inverse of the continuous additive functional
\begin{displaymath}
\sum_{x\in \mathcal{G}_{n}} L^{x}_{t}(B^{\widetilde{\mathcal{G}}_{n}}).
\end{displaymath}
Let $n\geq r+1$. The occupation field $\Big(L^{y}_{\tilde{\tau}_{u}^{n}}(B^{\widetilde{\mathcal{G}}_{n}})\Big)_{y\in \widetilde{\Lambda}_{r}}$ is obtained from the set of discrete-space trajectories
\begin{equation}
\label{SecInterlacement: EqTraj1}
\left\lbrace(X^{n}_{T_{[-r,r]^{d},j}^{n}+t})_{0\leq t\leq R_{j}^{n}-T_{[-r,r]^{d},j}^{n}}\vert 
j\in J_{u,[-r,r]^{d}}^{n}\right\rbrace
\end{equation}
in the same way as the occupation field $\Big(L^{y}(\widetilde{\mathcal{I}}^{u})\Big)_{y\in \widetilde{\Lambda}_{r}}$ is obtained from the trajectories
\begin{equation}
\label{SecInterlacement: EqTraj2}
\lbrace (w^{[-r,r]^{d}}(t))_{t\geq 0}\vert w\in \mathcal{I}_{[-r,r]^{d}}^{u}\rbrace.
\end{equation}
In both cases one adds Brownian excursions and takes the local times. According to Lemma 
\ref{SecInterlacement: LemApproxScheme} applied to the set $[-r,r]^{d}$, \eqref{SecInterlacement: EqTraj1} converges in law to \eqref{SecInterlacement: EqTraj2}. This implies the convergence in law of $\Big(L^{y}_{\tilde{\tau}_{u}^{n}}(B^{\widetilde{\mathcal{G}}_{n}})\Big)_{y\in \widetilde{\Lambda}_{r}}$ to 
$\Big(L^{y}(\widetilde{\mathcal{I}}^{u})\Big)_{y\in \widetilde{\Lambda}_{r}}$.
\end{proof}

\begin{proposition}
\label{SecInterlacement: PropIsomorphismCable}
Let $(\phi_{y})_{y\in\widetilde{\mathbb{Z}}^{d}}$ be the Gaussian free field on the metric graph $\widetilde{\mathbb{Z}}^{d}$ and $(\phi'_{y})_{y\in\widetilde{\mathbb{Z}}^{d}}$ a copy of $\phi$ independent of $\widetilde{\mathcal{I}}_{u}$. The following equality in law holds:
\begin{equation}
\label{SecInterlacement: EqIsomorphismCable}
\Big(L^{y}(\widetilde{\mathcal{I}}^{u})+\dfrac{1}{2}\phi'^{2}_{y}\Big)_{y\in\widetilde{\mathbb{Z}}^{d}}\stackrel{(d)}{=}
\Big(\dfrac{1}{2}(\phi_{y}-\sqrt{2u})^{2}\Big)_{y\in\widetilde{\mathbb{Z}}^{d}}.
\end{equation}
\end{proposition}

\begin{proof}
Let $\phi^{n}$ be the Gaussian free field on the metric graph $\widetilde{\mathcal{G}}_{n}$ associated to the Brownian motion with instantaneous killing at $x_{\ast}$ ($\phi^{n}_{x_{\ast}}=0$). Let $\phi'^{n}$ be a copy of $\phi^{n}$ independent of the Brownian motion $B^{\widetilde{\mathcal{G}}_{n}}$ starting from $x_{\ast}$. The second generalized Ray-Knight theorem (see Theorem $8.2.2$ in \cite{MarcusRosen2006MarkovGaussianLocTime}) holds in this setting:
\begin{displaymath}
\Big(L^{y}_{\tilde{\tau}_{u}^{n}}(B^{\widetilde{\mathcal{G}}_{n}})
+\dfrac{1}{2}(\phi'^{n}_{y})^{2}\Big)_{y\in\widetilde{\mathcal{G}}_{n}}\stackrel{(d)}{=}
\Big(\dfrac{1}{2}(\phi^{n}_{y}-\sqrt{2u})^{2}\Big)_{y\in\widetilde{\mathcal{G}}_{n}}.
\end{displaymath}
Since the $\phi^{n}$ converges in law to $\phi$ and according to Lemma \ref{SecInterlacement: LemApproxCable}
$\Big(L^{y}_{\tilde{\tau}_{u}^{n}}(B^{\widetilde{\mathcal{G}}_{n}})\Big)_{y\in\widetilde{\mathcal{G}}_{n}}$ converges in law to $\Big(L^{y}(\widetilde{\mathcal{I}}^{u})\Big)_{y\in \widetilde{\mathbb{Z}}^{d}}$ we get the isomorphism \eqref{SecInterlacement: EqIsomorphismCable}.
\end{proof}

\begin{proof}[Proof of Theorem \ref{ThmLevelSets}]
The coupling is the following: take a discrete-space random interlacement $\mathcal{I}^{u}$ and extend it to a continuous interlacement $\widetilde{\mathcal{I}}^{u}$ of the metric graph $\widetilde{\mathbb{Z}}^{d}$. Take a Gaussian free field $\phi'$ on $\widetilde{\mathbb{\mathbb{Z}}}^{d}$ independent on $\widetilde{\mathcal{I}}^{u}$. Using isomorphism
\eqref{SecInterlacement: EqIsomorphismCable} we see $\Big(L^{y}(\widetilde{\mathcal{I}}^{u})
+\frac{1}{2}\phi'^{2}_{y}\Big)_{y\in\widetilde{\mathbb{Z}}^{d}}$ as 
$\Big(\frac{1}{2}(\phi_{y}-\sqrt{2u})^{2}\Big)_{y\in\widetilde{\mathbb{Z}}^{d}}$ where $\phi$ is a Gaussian free field on $\widetilde{\mathbb{\mathbb{Z}}}^{d}$ and sample the sign of $\phi-\sqrt{2u}$ using its conditional law given $\vert\phi-\sqrt{2u}\vert$.

The continuous occupation field $\Big(L^{y}(\widetilde{\mathcal{I}}^{u})\Big)_{y\in \widetilde{\mathbb{Z}}^{d}}$ is strictly positive on all the vertices and inside the edges visited by the discrete random interlacement $\mathcal{I}^{u}$. In the isomorphism \eqref{SecInterlacement: EqIsomorphismCable}, $(\vert\phi_{y}-\sqrt{2u}\vert)_{y\in\widetilde{\mathbb{Z}}^{d}}$ is strictly positive on these vertices and inside these edges. This means that each trajectory in $\mathcal{I}^{u}$ is contained in a sign cluster of $\phi-\sqrt{2u}$, which is necessarily unbounded. But according Proposition \ref{PropFiniteClusters},
$\phi$ has only bounded sign clusters on the metric graph and \textit{a fortiori} the connected components of 
$\lbrace y\in\widetilde{\mathbb{Z}}^{d}\vert \phi_{y}>\sqrt{2u}\rbrace$ are all bounded. Thus in our coupling all the vertices visited by $\mathcal{I}^{u}$ are contained in $\lbrace y\in\widetilde{\mathbb{Z}}^{d}\vert \phi_{y}<\sqrt{2u}\rbrace$ and since these are vertices, they are contained in $\lbrace x\in \mathbb{Z}^{d}\vert \phi_{x}<\sqrt{2u}\rbrace$.
\end{proof}

The fact that for all $h>0$, $\lbrace x\in \mathbb{Z}^{d}\vert \phi_{x}<h\rbrace$, seen as a dependent site percolation on $\mathbb{Z}^{d}$, has an infinite cluster was proved in \cite{BricmontLebowitzMaes1987PercGFF}. However, Theorem \ref{ThmLevelSets} may be used as an alternative proof of this fact.

\section*{Acknowledgements}
I would like to thank Yves Le Jan and Alain-Sol Sznitman
for fruitful discussions. I would also like to thank the anonymous reviewer for
his detailed and thorough review. 
I would like to thank Hao Wu for pointing out an error in
Formula \eqref{SecCouplComplex: EqLinInterpol} in a previous version,
and
Franco Severo for pointing out that in the proof of
Proposition \ref{PropAltProof}, an argument was missing for the extension to the case of
infinite graph.

\bibliographystyle{plain}
\bibliography{tituslupuabr2}

\end{document}